\documentclass[11pt]{article}
\usepackage{amssymb,amscd,amsthm,amsmath,fancybox,float,epsfig}
\usepackage{afterpage}
\usepackage{appendix}
\usepackage{hyperref}
\usepackage{graphicx}
\usepackage{caption}
\usepackage{subcaption}
\usepackage{times}
\usepackage{authblk}

\newcommand{\norm}[2]{\left\| #1 \right\|_{#2}}

\newcommand{\abs}[1]{\left| #1 \right|}

\DeclareMathOperator*{\argmin}{arg\,min}

\newcommand\cR{\mathcal R}

\newcommand\bc{\boldsymbol c}

\newcommand\ba{\boldsymbol a}

\newcommand\bzero{\boldsymbol 0}
\newcommand\balpha{\boldsymbol \alpha}

\newcommand\bBeta{\boldsymbol \beta}
\newcommand\bj{\boldsymbol j}

\newcommand\bk{\boldsymbol k}
\newcommand\bv{\boldsymbol v}
\newcommand\bu{\boldsymbol u}
\newcommand\bn{\boldsymbol n}

\newcommand\brho{\boldsymbol \rho}
\newcommand\bx{\boldsymbol x}
\newcommand\bw{\boldsymbol w}

\newcommand\by{\boldsymbol y}

\newcommand\cF{\mathcal F}

\newcommand\bz{\boldsymbol{z}}
\newcommand\Bz{\boldsymbol{z}}

\newcommand\ta{\widetilde{a}}

\newcommand\hf{\widehat{f}}
\newcommand\hg{\widehat{g}}

\newcommand\hrho{\widehat{\rho}}

\newcommand\TA{\mathbb A}

\newcommand\cA{\mathcal{A}}

\newcommand\cS{\mathcal{S}}

\newcommand\halpha{\hat \alpha }

\newcommand\bomega{\boldsymbol{\omega}}

\renewcommand\Re{\operatorname{Re}}
\renewcommand\Im{\operatorname{Im}}

\newcommand\bbE{\mathbb E}
\newcommand\bbC{\mathbb C}

\newcommand\bbN{\mathbb N}

\newcommand\bbR{\mathbb R}
\newcommand\bbZ{\mathbb Z}

\newcommand\supp{\operatorname{supp}}

\newtheorem{theorem}{Theorem}

\theoremstyle{definition}
\newtheorem{definition}{Definition}

\newtheorem{example}{Example}

\theoremstyle{remark}

\begin{document}
\title{Recovering missing data in coherent diffraction imaging}
\author[1]{D.~A. Barmherzig}
\author[1]{A.H.~Barnett}
\author[1,2]{C.L.~Epstein}
\author[1,3]{L.F.~Greengard}
\author[1]{ J.F.~Magland}
\author[1]{M.~Rachh}
\affil[1]{Center for Computational Mathematics, Flatiron Institute,
  New York, NY}
\affil[2]{Department of Mathematics,  University of Pennsylvania}
\affil[3]{Courant Institute, New York University}
\maketitle

\begin{abstract}
  In coherent diffraction imaging (CDI) experiments, the intensity of
  the scattered wave impinging on an object is measured on an array of
  detectors. This signal can be interpreted as the square of the
  modulus of the Fourier transform  of the unknown
  scattering density.  A beamstop obstructs
  the forward scattered wave and, hence, the modulus Fourier data from
  a neighborhood of $\bk=\bzero$ cannot be measured. In this note, we
  describe a \emph{linear} method for recovering this unmeasured modulus
  Fourier data from the measured values and an estimate of the support
  of the image's autocorrelation function {\em without} consideration
  of phase retrieval. We analyze the effects of noise, and the conditioning of this problem,
  which grows exponentially with the modulus of the maximum spatial
  frequency not measured.\newline
\end{abstract}

\section{Introduction}
In coherent diffraction imaging (CDI) experiments, the intensity of the
scattered wave impinging on an object is measured on an array of detectors. This
signal can be interpreted as the square of the modulus of the Fourier transform
$|\hrho(\bk)|^2$ of the unknown scattering density, denoted by $\rho(\bx)$ \cite{Miao1999, Chapman2006}.  The
spatial frequency, $\bk,$ is related to the scattering direction through the Ewald
sphere construction \cite{chapmanetal}. We assume here that the x-ray wavelength
is sufficiently small that the curvature of the Ewald sphere can be neglected
and that we are sampling $|\hrho(\bk)|^2$ on a uniform grid.
The {\em phase retrieval problem} is to recover the complex
values $\hrho(\bk)$, and hence the desired unknown, from the measured intensity
data, supplemented by auxiliary information, which is typically the approximate
support of $\rho(\bx)$ \cite{GeomPR18,EldarReview15,Barmherzig2019}. Unfortunately, a {\em beamstop} obstructs the forward
scattered wave and, hence, the modulus Fourier data from a neighborhood of
$\bk=\bzero$ cannot be measured. Standard iterative approaches to 
recovering the unmeasured samples, such as HIO or difference maps 
\cite{ERTPNAS2007,Fienup1982},
use the auxiliary
information to fill in the unmeasured Fourier coefficients at the same
time as the image itself is reconstructed.  
In this note, we describe a \emph{linear}
method for recovering this unmeasured modulus Fourier data from the measured
values and an estimate of the support of the image's autocorrelation function,
{\em without} consideration of phase retrieval.

To set various parameters and length scales, we assume 
that $\rho$ is supported in a
compact subset of ${\cS} = (-\frac 12,\frac 12)^d,$ and that 
its autocorrelation image
\[
(\rho\star\rho)(\by) \equiv \int_{\cal S} \rho(\bx) \rho(\by+\bx) \, d\bx
\]
is supported in 
$(-\frac \beta 2,\frac \beta 2)^d.$ We define the {\em field of view}
(FOV)
as the box
$(-\frac m 2,\frac m 2)^d$ (Fig. \ref{figsetup}).

The Fourier transform, $\hrho,$ is defined by
\begin{equation}
  \hrho(\bk)\equiv \int_{\bbR^d}\rho(\bx)e^{-2\pi i\bx\cdot\bk}d\bx,
\end{equation}
so that 
\begin{equation}
  \rho(\bx)\equiv \int_{\bbR^d}\hrho(\bk)e^{2\pi i\bx\cdot\bk}d\bk.
\label{invtrans}
\end{equation}
Finally, we assume that $|\hrho(\bk)|^2$ is given in the box ${\cal D}$ 
of side length $2 K_{max}$, from which a window 
${\cal W} = [-k_0,k_0]^d$ is deleted,
corresponding to the beamstop (Fig. \ref{figsetup}).

\begin{figure}[H]
  \centering
       \includegraphics[width=12cm]{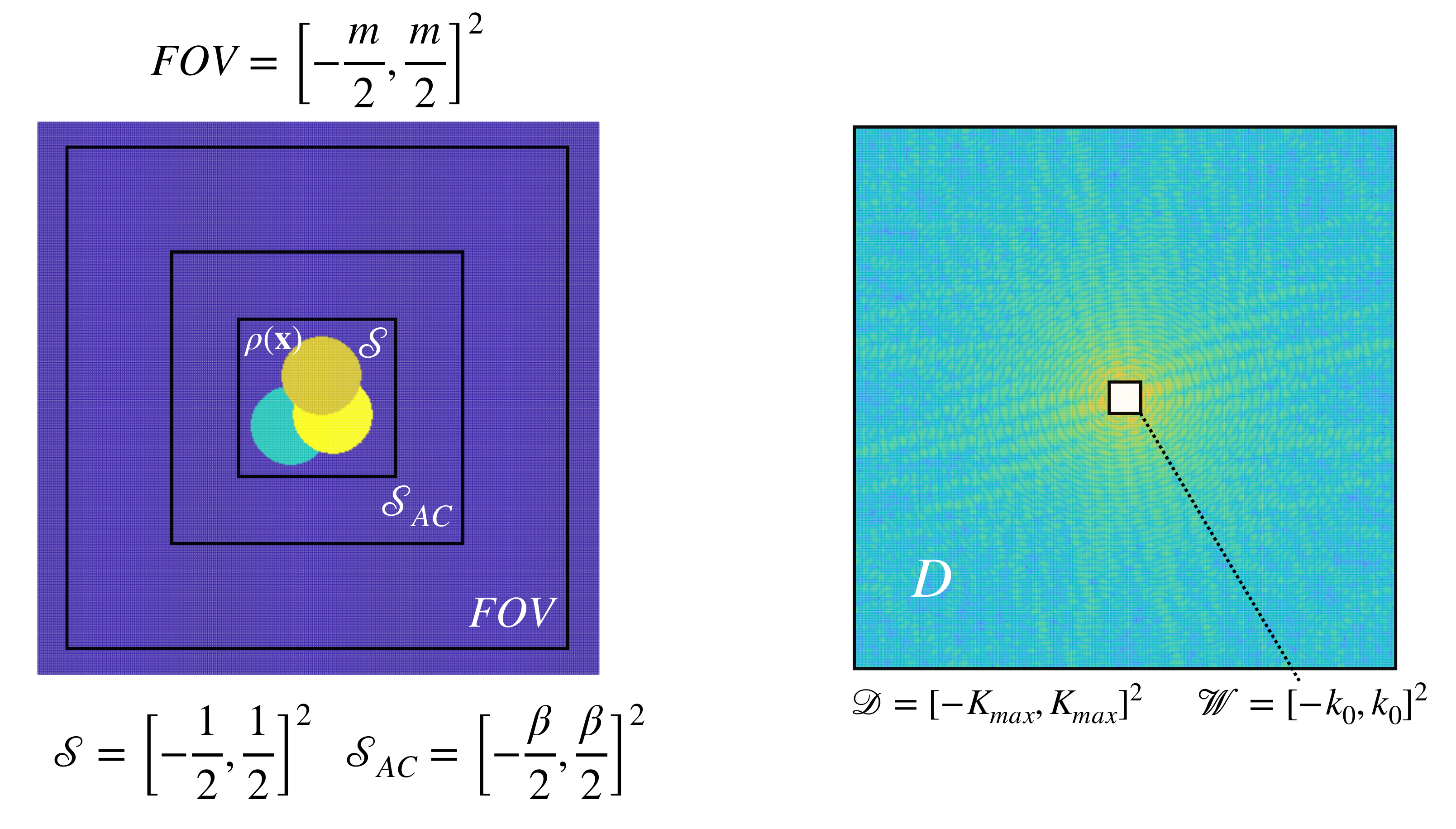}
        \caption{A two-dimensional 
object of interest $\rho(\bx)$ is supported in the
bounded region ${\cal S} = [-\frac 12,\frac 12]^2$  and its autocorrelation is
supported in ${\cal S}_{AC} = [-\frac \beta 2,\frac \beta 2]^2 $. 
The field of view (FOV) is the maximum region in physical space where 
we expect the Fourier
transform \eqref{invtrans} to be valid for a given sampling of 
$\hrho(\bk)$. We denote by $R$ the subset of the field of view outside
${\cal S}_{AC}$. In the transform domain ($k$-space), 
the modulus of $\hrho(\bk)$ is measured on a box ${\cal D}$, with the region
${\cal W}$ obscured by the beamstop.
}
\label{figsetup}
\end{figure}

For the sake of simplicity, we work in the discrete setting, with 
\begin{equation}\label{eqn1}
  \rho_{\bj}=\rho\left(\frac{\bj}{2N}\right),  \text{ for }\quad\bj\in
[1-N:N]^d.
\end{equation}
From the Nyquist sampling theorem, the grid spacing $\frac{1}{2N}$ in
physical space corresponds to a maximum frequency in the transform
domain of $K_{max} = N$.  Following Fig. \ref{figsetup}, we define
$J=[1-mN:mN]^d,$ in order to cover the field of view.  The vector
$\brho\in \bbR^J,$ is of length $(2mN)^d,$ and has entries defined by
\eqref{eqn1} for $\bj\in [1-N:N]^d$ and zero otherwise, corresponding
to the fact that $\rho(\bx)$ is supported in $\cS.$ In taking the
discrete Fourier transform of $\brho$ with a fixed $N$, increasing
values of $m$ lead to a finer sampling of $\hrho$ in the transform
domain, without changing the maximum frequency $K_{max} = N$.  As a
result, we sometimes refer to $m>1$ as the ``oversampling'' factor.
In the CDI experiment, oversampling corresponds to using an array of
sensors that measure $|\hrho(\bk)|$ on a grid with spacing $\Delta k =
1/m$.  This is the Nyquist sampling rate in the inverse direction,
sufficient to recover a physical object {\em within the field of
  view}.  It is the combination of oversampling with prior information
about the support of $\rho(\bx)$ that makes the phase retrieval
problem solvable, for a dense open set of data, in dimensions $d>1$
\cite{BruckSodin1979,Hayes1982}.

In the remainder of this paper, we let
\begin{equation}\label{eqn2.0}
  \hrho_{\bk}\equiv \sum_{\bj \in J}\rho_{\bj}e^{-\frac{2\pi i \bj\cdot\bk}{2mN}}\approx
   (2N)^d \hrho\left(\frac{\bk}{m}\right)
\end{equation}
denote the discrete Fourier transform (DFT) of the data extended
by zero to the entire field of view. Thus, in our model
problem, the measured intensity is proportional to
$(|\hrho_{\bk}|^2)$. 
We also recall the well-known fact that
$(|\hrho_{\bk}|^2)$ is the DFT of the discrete
(periodic) autocorrelation image:
\[
(\brho\star\brho)_{\bk}\equiv \sum_{\bj\in J}\rho_{\bj}\rho_{\bj+\bk}.
\]

Let $W\subset J$ denote the set of lattice points obstructed by the
beamstop ${\cal W}$.  Substantial effort has been devoted to the
development of methods for approximating the Fourier coefficients at
these frequencies.  Typically, this involves an iterative method
designed to solve the phase retrieval problem and missing data problem
simultaneously (see, for example, \cite{NishinoMiaoIshikawa2003, Cossairt2015}.)  In
Section~\ref{s.rec_alg} we describe a linear algorithm for recovering
the unmeasured values $\{|\hrho_{\bk}|^2:\:\bk\in W\}.$ It amounts to
solving a least squares problems for the unmeasured coefficients,
using knowledge about the support of the autocorrelation of $\brho$ as
a constraint. In Section~\ref{sec3.01} we analyze the conditioning of
this problem, by relating it to classical results for prolate
spheroidal functions, see~\cite{Slepian1978}.  We show that if $|W|$
is not too large, then, with sufficiently fine sampling in the Fourier
domain, e.g. $m\geq 3,$ the unmeasured magnitude data,
$\{|\hrho_{\bk}|^2\text{ for }\bk\in W\},$ can be stably determined by
solving the least squares problem (Fig. \ref{fig:beamstop}).

While ${\cal W}$ (and ${\cal S}_{AC}$) 
can, in principle take any shape, 
we assume for simplicity that it is square so that the lattice points
lying within ${\cal W}$ are of the form
$W=[1-w:w-1]^d$, where $w = \lfloor{ 1+  m k_0} \rfloor$.
With the discretized autocorrelation image 
$\brho\star\brho$ supported in $[-\beta N:\beta N]^d,$ 
connections with prolate spheroidal functions show that
asymptotically, as $m, N,$ and $\beta k_{0}$ grow large, the
conditioning of the linear method grows like
\begin{equation}\label{eqn4.003}
 \kappa(\beta,k_{0})\sim \frac{e^{\pi \beta k_{0}}}{\sqrt{4\pi
     d}[\beta k_{0}]^{\frac 14}}.
\end{equation}
This estimate is proven in Appendix~\ref{A1}.
Empirically this asymptotic result is already accurate for
$N\geq 128, m\geq 3.$ The main lessons of~\eqref{eqn4.003}, and the
examples in Section~\ref{exmpl.sec}, are
\begin{enumerate}
\item The conditioning of the hole filling problem
problem is largely determined by the \emph{physical} space-bandwidth product
$\beta k_{0}.$
\item Little can be gained in this context by taking either $m\geq 4,$ or $N$ very large
  (neither parameter appears in the formula).
  \item As the dimension, $d,$  increases the conditioning can be expected
    to slightly improve; in particular, the exponent
    in~\eqref{eqn4.003} does \emph{not} depend
    on $d.$
\end{enumerate}
 
\begin{figure}[H]
  \centering
       \includegraphics[width=12cm]{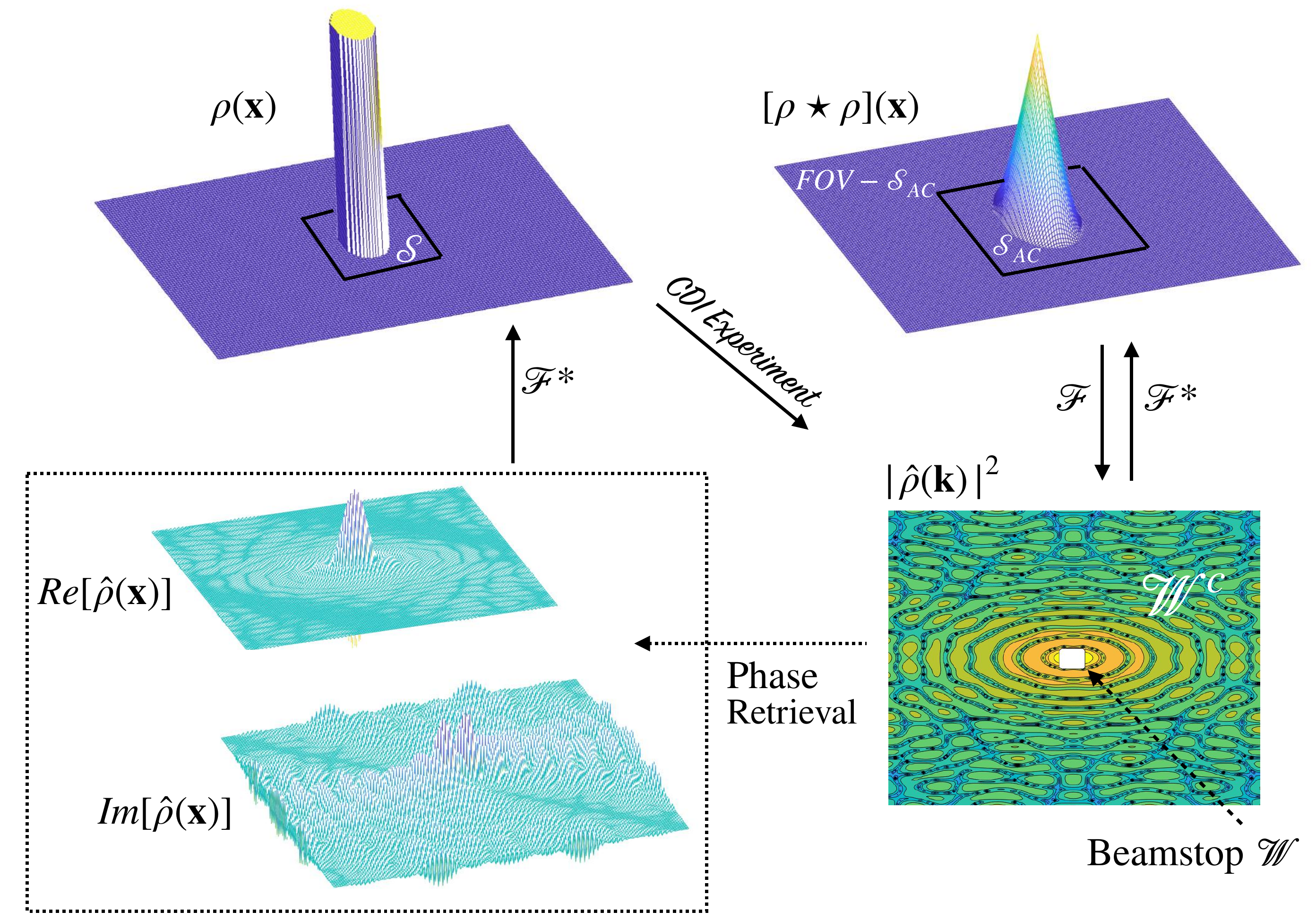}
        \caption{An unknown object $\rho(\bx)$ is supported in a bounded region 
          ${\cal S}$ (upper left) and its autocorrelation is supported in 
        ${\cal S}_{AC}$ (upper right).
         In standard methods for phase retrieval
         from CDI experiments, 
         the missing data within the beamstop ${\cal W}$
         is inferred as part of an overall iterative scheme. 
         Here, we solve for the missing data itself, without consideration
         of phase, by using an estimate for the support of the 
         autocorrelation image
         and solving a linear least squares problem.}\label{fig:beamstop}
\end{figure}

The effects of noise are analyzed in Sections~\ref{noise.sec}--\ref{sec6}, where it is shown
that, if $\beta k_{0}$ is not too large, then, with sufficient SNR, this
scheme can be robust even in the presence of noise.  At lower SNR, we
show that improved images may result if some of the reconstructed
modulus data is used, and some of the coefficients are found
implicitly in the phase retrieval step.

\section{The Recovery Algorithm}\label{s.rec_alg}
 In our model, the measured data, denoted by $\ba^2,$ consists of
\begin{equation}\label{eqn4.01}
  a^2_{\bj}=\begin{cases} |\hrho_{\bj}|^2&\text{ for }\bj\in  W^c = J\setminus
  W\\
  0&\text{ for }\bj\in   W.
  \end{cases}
\end{equation}
In the image domain,
let $S \subset [1-N:N]^d \subset J$ be the lattice points
within our estimate for the support of $\brho$, then
$$ S_{AC} = S\ominus S=\{\bj-\bk:\: \bj,\bk\in S\},$$
is an estimate for the support of the autocorrelation image
${\cal S}_{AC},$ and set $R= J\setminus S_{AC}$.  If $|R|>|W|$, then,
in principle, the unmeasured magnitude data can be determined.  For
this problem to be reasonably well conditioned the ratio $|W|/m^d$
must be sufficiently small, and $|R|>> |W|.$ Empirically, a little
more oversampling ($m=3$) than is required for the phase retrieval
problem to be solvable ($m=2$) produces markedly better results.
Having more samples also leads to better noise reduction when
recovering the unmeasured samples.  On other hand, greater
oversampling may require a smaller pixel size on the detector, or a
more distant detector, either of which would tend to increase the
noise content of individual measurements, so clearly there are
trade-offs to be considered. Our asymptotic analysis, and numerical
examples indicate that there is little improvement beyond $m=4.$

Let $\cF$ denote the $d$-dimensional DFT matrix, normalized to be a
unitary operator, and let $\cF^*$ be its adjoint.  To keep the
notation simpler, we omit the spatial dimension $d$ when the context is
clear. We interpret
$\cF$ as a map from data on the $|J|$-point grid in the physical
domain to a $|J|$-point grid in the frequency domain, both contained
in $\bbZ^d.$ The frequency domain grid is normalized to be centered on
$\bk=\bzero.$ In the remainder of the paper we let
$\hrho_{\bk}=[\cF(\brho)]_{\bk},$ which differs, by the constant factor,
$[2mN]^{-\frac{d}{2}},$ from the normalization in~\eqref{eqn2.0}.

\begin{definition}
We denote by $\cF_{W,R}$ the submatrix of 
of $\cF$ that maps data from grid points in $R$ to Fourier
transform points in $W$. $\cF_{W,S_{AC}}$ is the submatrix that maps
data from grid points in $S_{AC}$ to Fourier
transform points in $W$. $\cF_{W^c,R}$ and $\cF_{W^c,S_{AC}}$ are defined
in the same manner, as are the submatrices of the adjoint:
$\cF^*_{R,W}$, $\cF^*_{S_{AC},W}$, $\cF^*_{R,W^c}$, $\cF^*_{S_{AC},W^c}$.
\end{definition}

\noindent
Note that taking the adjoint interchanges the roles of the two
subsets, e.g., $[\cF_{W,R}]^*=\cF^*_{R,W}.$


As noted above, the DFT coefficients of the autocorrelation
image, $\brho\star\brho,$ are $\{|\hat{\rho}_{\bj}|^2:\:\bj\in J\}$ (the
Wiener-Khinchin theorem).  Let us now write the inverse DFT in block
form:
\begin{equation}
\left(
\begin{array}{cc}
\cF^*_{R,W} & \cF^*_{R,W^c} \\
\cF^*_{S_{AC},W} & \cF^*_{S_{AC},W^c}
\end{array}
\right)
\left(
\begin{array}{c}
\balpha_W \\
\ba^2_{W^c}
\end{array}
\right)
=
\left(
\begin{array}{cc}
0 \\
\brho \star\brho 
\end{array}
\right),
\label{idftsys}
\end{equation}
where $\ba^2_{W^c} = \ba^2$ restricted to $W^c$, is the measured data and
$\balpha_W$ denotes the (unmeasured) coefficients $\alpha_{\bj}$ of $\ba^2$ for $\bj$ restricted to $W$.
Clearly, letting $\balpha_W = (|\hat{\brho}_j|^2)_{\bj\in W}$  yields
a consistent solution of
\eqref{idftsys}, since this is simply a restatement of the Wiener-Khinchin theorem. 
If we restrict our attention to the first row, we have the $|R| \times |W|$ linear 
system:
\begin{equation}
\cF^*_{R,W} \balpha_W = 
- \cF^*_{R,W^c} \ba^2_{W^c}.
\label{lsqprob}
\end{equation}
This is shown schematically in Figure~\ref{fig:beamstop}.

For small sets $W$ and large sets $R$ the system of equations
$\cF^*_{R,W}\balpha_W=0$ has only the trivial solution
$\balpha_W=\bzero.$ Assuming that the data $\ba^2_{W^c}$ is exact,
then the highly overdetermined system in \eqref{lsqprob} has the exact
solution, $\balpha_W = (|\hat{\brho}_j|^2)_{\bj\in W},$ which is
unique. For generic right hand sides, the equation $\cF^*_{R,W} \bx= -
\cF^*_{R,W^c} \by,$ does not have an exact solution, and in the
remainder of the paper we take $\bx$ to be the solution to the least
squares problem:
\begin{equation}
  \bx_0=\argmin_{\bx}\|\cF^*_{R,W} \bx + \cF^*_{R,W^c} \by\|_2,
\end{equation}
which is also unique, as $\cF_{W,R}\cF^*_{R,W}$ is invertible.
More precisely, we have

\begin{theorem}\label{thm0}
Suppose that $J=[1-M:M]^d$. If $W \subset [p:p+u]\times [1-M:M]^{d-1}$, and
$R\supset[q:q+v]\times [1-M:M]^{d-1}$, with  $v>u$,
then $\cF^*_{R,W}\bx=0$ has only the trivial solution.
\end{theorem}

\begin{proof}
Let $\bx\in \bbR^J$ be a vector, with support in $W,$ that belongs to
the the null-space of $\cF^*_{R,W},$ and let $X(\bz)$ be its
$Z$-transform.  For every frequency $\bk$, the adjoint DFT,
$\check{\bx}_{\bk},$ equals $ X(\bomega_{\bk})$, for $\bomega_{\bk}$
an appropriate vector of points on the torus $(S_1)^d$.  We can
rewrite the $Z$-transform as
\[ X(\bz)=\sum_{\bj'\in [1-M:M]^{d-1} } p_{\bj'}(z_1)(z_2,\dots,z_d)^{\bj'}.
\]
Up to a factor of $z_1^p$, each $p_{\bj'}(z_1)$ is a polynomial of degree $u$. 

The hypothesis of the theorem implies that for any 
$\bomega_{\bk}=(\omega_{k_1},\dots,\omega_{k_d})$ with $k_1\in [q:q+u]$, 
we have that $X(\bomega_{\bk})=0$.
By the invertibility of the $(d-1)$-dimensional DFT, this implies that
$p_{\bj'}(\bomega_{k_1})=0$ for all $\bj'$ and $k_1\in [v:q+v]$. 
Because $v>u$ and $p_{\bj'}$ are polynomials of degree $u$, 
this shows that the polynomials are actually all zero, which, 
in turn, implies that $\bx=\bzero$ as well.
\end{proof}

As $|W|$ is a reasonably small number, the reduced SVD of 
$$\cF^*_{R,W} =U\Sigma V^*$$
is fairly easy to compute.  The Moore-Penrose inverse of $\cF^*_{R,W}$
is
\begin{equation}
 \cF^{*\dag}_{R,W}=V\Sigma^{-1} U^*.
\end{equation}
The unique solution to the overdetermined linear
system in~\eqref{lsqprob} is given by
\begin{equation}\label{eqn8.02}
   \balpha_W= - \cF^{*\dag}_{R,W}\cF^*_{R,W^c} \ba^2_{W^c}.
\end{equation}
We call the operator
\begin{equation}\label{eqn9.02}
  \cR_{R,W}=- \cF^{*\dag}_{R,W}\cF^*_{R,W^c}
\end{equation}
the \emph{recovery operator}.  For general right hand sides, $\by,$
the solution to the least squares problem is given by $\cR_{R,W}\by.$
It should be noted that the recovery operator only depends on $W,J,R,$
and is independent of the particular image being reconstructed.

\begin{figure}[H]
  \centering
       \includegraphics[width=12cm]{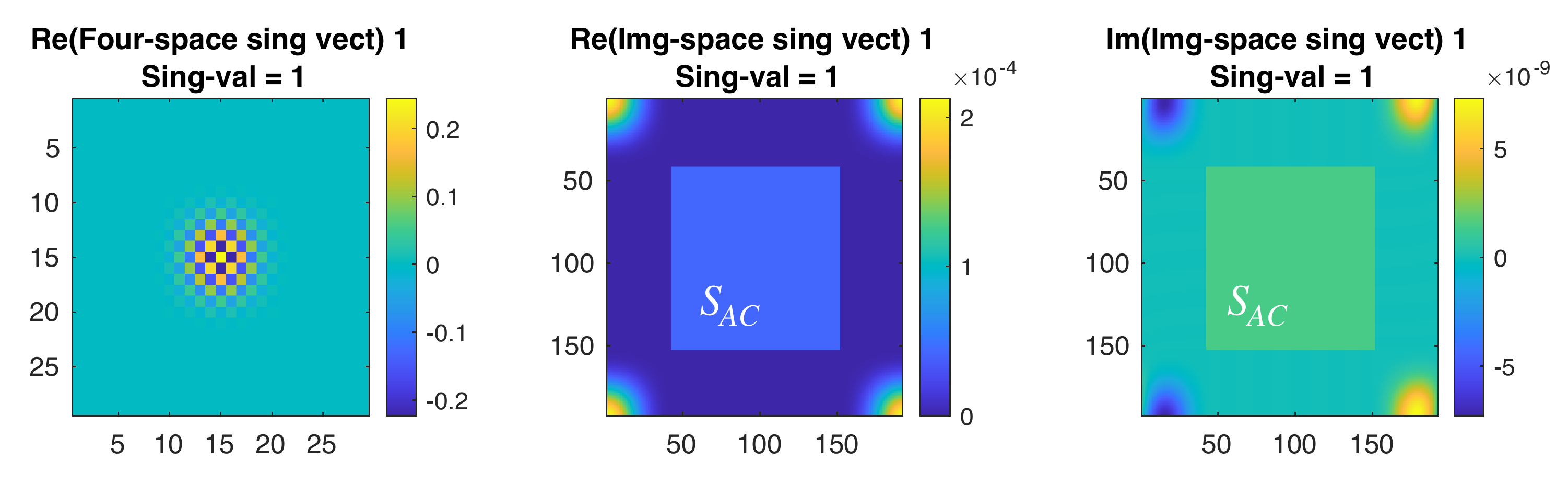}
        \caption{The singular vector $\bu_1$ of $\cF^*_{R,W}$,
          corresponding to the largest
          singular value 1. On the left is the DFT representation, 
          showing a small
          neighborhood of $W$. In the middle is a plot of
          $\Re(\cF^*(\bu_1))$ and on the right is a plot of 
          $\Im(\cF^*(\bu_1)).$
          The set $S_{AC}$ is indicated in the middle and
          right panels as a lightly shaded rectangle.}\label{sv1fig}
\end{figure}

       \begin{figure}[H]
        \centering
       \includegraphics[width=12cm]{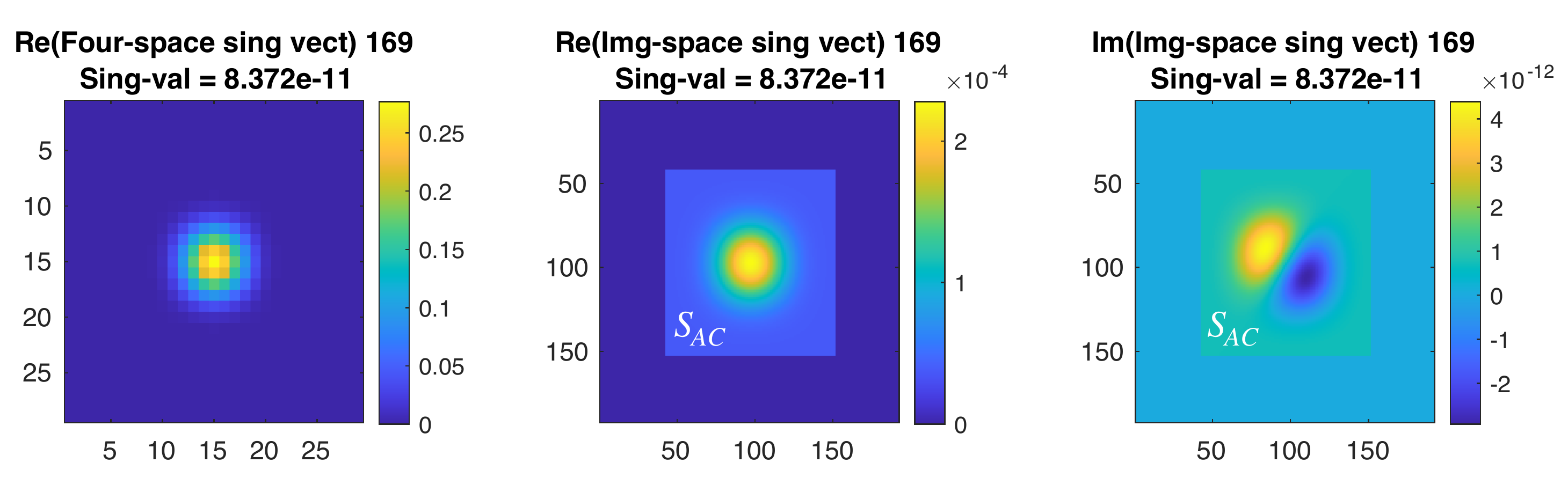}
        \caption{The singular vector
          $\bu_{169}$ of $\cF^*_{R,W}$, corresponding to the smallest
          singular value $9.15\times10^{-6 }.$ On the left is the DFT
          representation, showing a small neighborhood of $W.$
          In the middle is a plot of $\Re(\cF^*(\bu_1))$ and on
          the right is a plot of 
          $\Im(\cF^*(\bu_1)).$ The set $S_{AC}$ is indicated
          in the middle and
          right panels as a lightly shaded rectangle.}\label{sv0fig}
       \end{figure}

Since $\cF^*_{R,W}$ is the composition of the unitary map $\cF^*$ with
orthogonal projections, its singular values lie between $0$ and $1.$
It is straightforward to describe the sorts of images that lead to
singular vectors with singular values very close to 1, or very close
to 0.  In order for $\bu\in\bbC^J$ to satisfy $|\cF^*_{R,W}\bu|
\approx |\bu|,$ it is necessary for $\bu$ to be supported in $W$ and
for $\cF^*(\bu)$ to be almost entirely supported in $R.$ An example is
shown in Figure~\ref{sv1fig}.  The larger $R$ is, the easier it is to
find such images.

On the other hand, for $\cF^*_{R,W}\bu\approx \bzero$ it is necessary
for $\bu$ to be supported in $W$ and $\cF^*(\bu)$ to be supported
almost entirely in $J\setminus R.$ In $2d,$ these images resemble
tensor products of sampled Hermite functions. For a fixed $W$, such
vectors become more plentiful as $R$ gets smaller. An example is shown
in Figure~\ref{sv0fig}. For these examples we use a thrice oversampled
$192\times 192$ grid; $W$ is a $13\times 13$ square centered on
$\bk=(0,0),$ and $|R|=24,765.$ In most practical examples the largest
singular value of $\cF^*_{R,W}$ is very close to $1.$ In this example,
the ratio of the largest to smallest singular value of $\cF^*_{R,W}$
is $=1.0929\times 10^5.$ This quantity represents the conditioning of
the problem of recovering the samples of magnitude DFT in $W,$ and is
also the norm of $\cF^{*\dag}_{R,W}.$ In Section~\ref{sec3.01} we give
estimates and asymptotic results for the conditioning of this problem.

From Fig.~\ref{sv0fig}, we see that the singular vector with
the smallest singular value is essentially a Gaussian centered
at $\bzero.$ In fact, this vector turns out to provide the most
important contribution to ``filling the hole'' in $\bk$-space.
This is easily understood in terms of the continuum model
embodied in equations~\eqref{eqn1} and~\eqref{eqn2.0}. Since
$\rho$ is compactly supported, its Fourier transform is smooth
and has a Taylor expansion about zero,
$\hrho(\bk)=\hrho(\bzero)+\langle
\nabla\hrho(\bzero),\bk\rangle+\frac{1}{2} \langle
H_{\hrho}(\bzero)\bk,\bk\rangle+O(\|\bk\|^3),$ where
$H_{\hrho}(\bzero)$ is the matrix of second derivatives of
$\hrho$ at $\bzero.$ For $\rho$ a real valued function this
implies that
\begin{equation}
|\hrho(\bk)|^2=|\hrho(\bzero)|^2\exp(-\langle B\bk,\bk\rangle)+O(\|\bk\|^4),
\end{equation}
where
\begin{equation}
\langle B\bk,\bk\rangle=\frac{1}{|\hrho(\bzero)|^2}\left[|\langle\nabla\hrho(\bzero),\bk\rangle|^2-\hrho(\bzero)\langle
H_{\hrho}(\bzero)\bk,\bk\rangle\right].
\end{equation}
For the sort of functions that arise in CDI, the zero Fourier coefficient
$\hrho(\bzero)$ is much larger than any other. The analysis above
shows that, near to $\bk=\bzero,$ the function $|\hrho(\bk)|^2$ strongly resembles a
Gaussian, as does the singular vector of $\cF^*_{R,W}$ with the
smallest singular value.  As we see in the next example, this singular
vector plays a dominant role in filling in the unmeasured magnitude DFT
data. 

\begin{example}
 Let $\{\bv_l:\: l=1,\dots 169\}$ denote the right singular vectors defined by
the matrix $\cF^*_{R,W}$ used in Figures~\ref{sv1fig} and~\ref{sv0fig}, with
the corresponding singular values $\{\sigma_l\}$ in decreasing
order. The solution to equation~\eqref{lsqprob} can then be represented as
\begin{equation}\label{eqn10}
\balpha_W=\sum_{l=1}^{169}c_j\bv_j.
\end{equation}
Figure~\ref{svcfig}[a] shows the coefficient vector $\bc,$ defined by
a non-negative image similar to those used in Example~\ref{exmpl2},
and Figure~\ref{svcfig}[b] shows the coefficient vector defined by an
image having both signs, but still having a large mean value. From
these plots it is quite apparent that $c_{169}$ is nearly an order of
magnitude larger than any other coefficient.
\begin{figure}[H]
\centering
\begin{subfigure}{.45\textwidth}
\includegraphics[width=6cm]{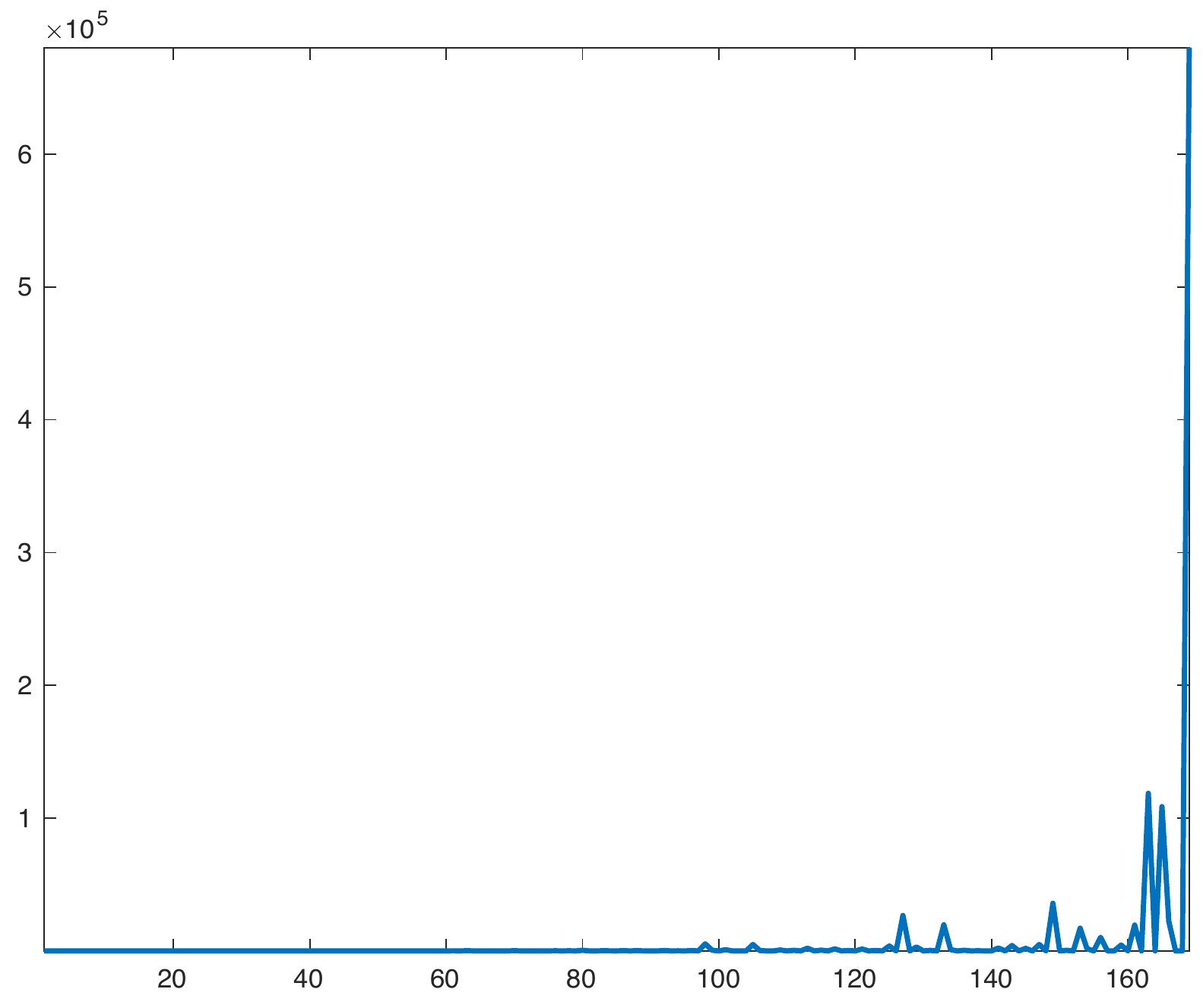}
\caption{Coefficients defined by a non-negative image.}
\end{subfigure}\qquad
\begin{subfigure}{.45\textwidth}
\includegraphics[width=6cm]{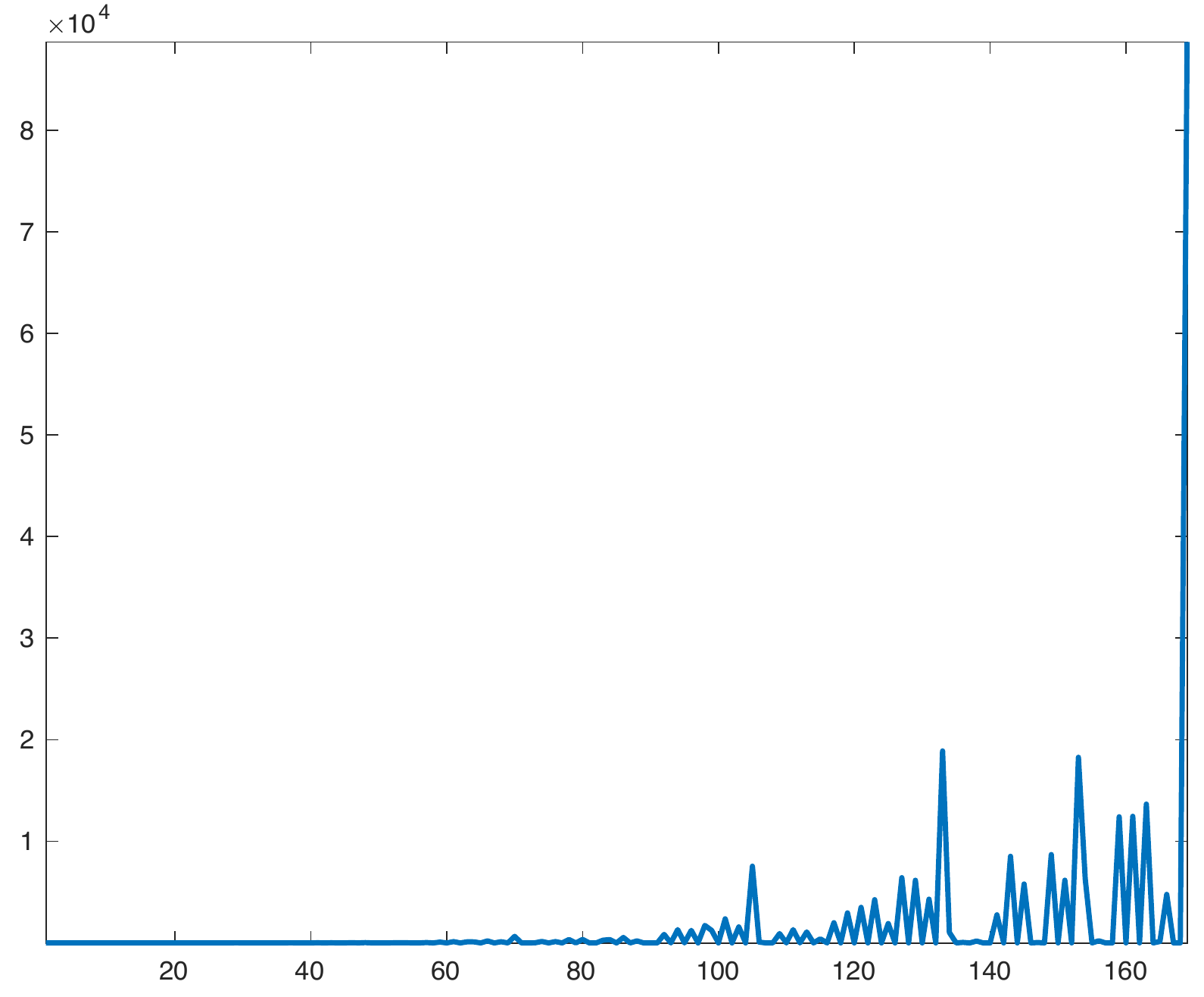}
\caption{Coefficients defined by an image with both signs.}
\end{subfigure}
\caption{The coefficient vectors from equation~\eqref{eqn10} defined by  real
images.}\label{svcfig}
\end{figure}
\end{example}

\section{The Norm of the Recovery Operator}\label{sec3.01}
 The recovery operator is defined in~\eqref{eqn9.02} as the
 composition of $\cF^{*\dag}_{R,W},$ the
Moore-Penrose inverse of
 $\cF^*_{R,W},$ with $\cF^*_{R,W^c}.$ The operator $\cF^*_{R,W^c}$ is
 a composition of orthogonal projections with the unitary operator
 $\cF^*,$ and therefore its norm is bounded by $1.$ Let $\sigma_1\geq
 \sigma_2\geq\cdots\geq\sigma_{|W|}$ denote the singular values of
 $\cF^*_{R,W}$ in decreasing order.
 The norm of $\cR_{R,W}$ is therefore bounded above by
 $\sigma_{|W|}^{-1},$ but, in fact, may be smaller.

 In this section, we restrict our attention to the case that $W$ is a
 square subregion of $J$ and $R$ is the \emph{complement} of the
 rectangular subregion $S_{AC}=S\ominus S$ within the field of view.
 Over the years, a great deal of effort has been expended to
 understand the singular values of operators like $\cF^*_{S_{AC},W},$
 a field of research that goes, at least in continuum case, under the
 rubric of ``prolate spheroidal functions,''
 see~\cite{SlepianPollack,Slepian1978,Fuchs1964}.  Because $\cF^*$ is
 a unitary map, and $R=S_{AC}^c,$ it follows that
\begin{equation}
  \|\balpha_W\|^2=\|\cF^*_{R,W}\balpha_W\|^2+\|\cF^*_{S_{AC},W}\balpha_W\|^2,
\end{equation}
and therefore:
\begin{equation}\label{eqn10.01}
\mu_0(R,W,d)=\sigma_{|W|}^2=  \min_{\balpha_W\neq
  \bzero}\frac{\|\cF^*_{R,W}\balpha_W\|^2}
   {\|\balpha_W\|^2}=
  1-\max_{\balpha_W\neq \bzero}\frac{\|\cF^*_{S_{AC},W}\balpha_W\|^2}{\|\balpha_W\|^2}.
\end{equation}
That is, there is a simple relationship between the smallest singular
value of $\cF^*_{R,W}$ and the largest singular value of
$\cF^*_{S_{AC},W}.$ In fact this is a special case of the following
theorem:
\begin{theorem}\label{thm1}
  Let $K, L\subset J$ and assume that $|K|\leq |L|.$ We let $\{\sigma_j\}$
  denote the singular values of $\cF^*_{L,K}$ in decreasing order
  and $\{\tau_{j}\}$ the singular values of $\cF^*_{L^c,K},$  also in
  decreasing order. If $p=|K|,$ then, for $1\leq j\leq p,$
  \begin{equation}
    \sigma_j^2=1-\tau_{p-j+1}^2.
  \end{equation}
\end{theorem}
\noindent
Note that $K,L$ are arbitrary subsets of $J$ subject to the
requirement that $|K|\leq |L|.$ The proof of the theorem is given in Appendix~\ref{A2}.

This theorem is very useful in the present setting,
where $W$ is a rectangular region and $R$ the
complement of a rectangular region: it allows us to reduce the
analysis of the singular values of $\cF^*_{R,W},$ to the case of
$\cF^*_{S_{AC},W}.$ 
Recalling that $w = \lfloor{ 1+ m k_0} \rfloor$,
with $W=[1-w:w-1]^d$ and $S_{AC}=[-\beta
  N:\beta N]^d,$ we let
$\{\tau_{j,d}\}$ be the singular values of $\cF^*_{S_{AC},W},$
in the $d$-dimensional case, and $\{\sigma_{j,d}\},$ the singular
values of $\cF^*_{R,W}.$ Because the $d$-dimensional DFT is the
$d$-fold tensor product of 1-dimensional transforms, it is not
difficult to show that
\begin{equation}
  \tau_{1,d}=\tau_{1,1}^d.
\end{equation}
If $q=|W|,$ then Theorem~\ref{thm1} implies that
$$\sigma_{q,d}=\sqrt{1-\tau_{1,d}^2}.$$
As follows from the analysis in Appendix~\ref{A1},
$\tau^2_{1,1}=(1-\epsilon),$ for an $\epsilon <\!<1,$ which
implies that
\begin{equation}
  \tau^2_{1,d}=(1-\epsilon)^d =1-d\epsilon+O(\epsilon^2),
\end{equation}
and therefore
\begin{equation}
  \sigma_{q,d}=\sqrt{1-\tau^2_{1,d}}\approx \sqrt{1-(1-d\epsilon)}\approx \sqrt{d\epsilon}.
\end{equation}
The norm of the operator of interest in the hole-filling-problem is
given approximately by:
\begin{equation}\label{eqn21.003}
[\mu_0(S_{AC},W,d)]^{-\frac 12}=  \sigma^{-1}_{q,d}\approx\frac{1}{\sqrt{d\epsilon}}.
\end{equation}

In Appendix~\ref{A1} we show how to get an asymptotic estimate for the quantity
$\mu_0(S_{AC},W,1),$ which depends only the ``space-bandwidth''
product, $\beta k_{0}.$ 
Asymptotically, as 
$\beta k_{0},m,N\to\infty,$ we show that
\begin{equation}\label{eqn11.007}
  \epsilon=\mu_0(S_{AC},W,1)\sim  4\pi \sqrt{\beta k_{0}}e^{-2\pi \beta k_{0}} .
\end{equation}
This formula, along with~\eqref{eqn21.003} imply the asymptotic formula:
\begin{equation}\label{eqn22.003}
 \|\cR_{W,R}\|\sim \frac{e^{\pi\beta k_{0}}}{\sqrt{4\pi d}[\beta
     k_{0}]^{\frac 14}}.
\end{equation}
Note that the exponent in~\eqref{eqn22.003} does
not depend on the dimension.

The recent analysis in~\cite{BarnettDFT2020} gives a lower bound,
which is slightly different, indicating that
increased sampling, and oversampling might have the effect of
slightly decreasing the norm of $\cR_{W,R}.$  A result of Slepian
(reproduced in~\cite{BarnettDFT2020})
shows, that as $m\to\infty,$
\begin{equation}
   \|\cR_{W,R}\|\sim
   \frac{C_{k_0,m,N}}{\sqrt{d}}\left[\frac{1+\tan\left(\frac{\pi\beta}{4m}\right)}
       {1-\tan\left(\frac{\pi\beta}{4m}\right)} \right]^{2mk_{0}+1},
\end{equation}
see~\cite{Slepian1978}. Here $C_{k_0,m,N}$ is an algebraic factor. 
As $m\to\infty$ this formula
gives the same exponential rate as~\eqref{eqn22.003}.

It is worth noting that the
exponential rate in the conditioning of the hole-filling problem does
\emph{not} depend on the dimension. In fact the condition number
should decrease, albeit slowly, as the dimension increases. The
computations in Example~\ref{exmpl3.01} show that~\eqref{eqn22.003}
is fairly accurate, even for moderate values of $\beta k_{0},$
$N,$ and  $m\geq 3.$ The main lessons of this analysis are:
\begin{enumerate}
\item The size of the hole in $\bk$-space that can be stably filled using
  the linear method we have introduced depends mostly on the product $\beta
  k_{0}.$
  \item The norm of recovery operator grows exponentially with this
    product. Recalling that $1<\beta<2,$ the size of the hole, as
    measured by $k_{0},$ that can be filled in this way is quite
    limited. However, the exponential rate does \emph{not} depend on the dimension!
    \end{enumerate}

\noindent
In the examples in the next section we see that, for a given $\beta$
and $k_{0},$ larger values of $m$ do provide a better result, though
with little improvement beyond $m=4.$

\section{Examples}\label{exmpl.sec}
We now consider several examples that illustrate the performance of
this method on $2d$-images, and the dependence of $\|\cR_{W,R}\|$ on
$m,N,k_{0},$ and $\beta.$ It should be recalled that oversampling
is a matter of changing the spacing between the samples collected in
$\bk$-space, and not the maximum frequency collected. As follows
from~\eqref{eqn2.0}, the double--oversampled Fourier coefficient with
indices $(2k_1,2k_2)$ is at the same spatial frequency as the
triple--oversampled coefficient with indices $(3k_1,3k_2).$

\begin{example}\label{exmpl2}
  For these examples we use an image, $\brho,$ taking both signs that
  sits in a $64\times 64$-rectangle. The function sampled is twice
  differentiable; for the estimate of the support $S,$ we use the
  1-pixel neighborhood of the smallest rectangle that contains
  $\supp\brho.$ We use either double, $|J|=128\times 128,$ or triple,
  $|J|=192\times 192,$ oversampling, and remove neighborhoods, $W,$ of
  $\bzero$ in $\bk$-space of various sizes.  In all cases we solve for
  the missing values using~\eqref{eqn8.02}.

  Figure~\ref{fig1} shows the results with double oversampling and
  Figure~\ref{fig2}, the results with triple oversampling.  The plots
  in the upper left corners show the singular values, in decreasing
  order, of $\cF^*_{R,W}.$ The plots in the upper right corners show
  the set $R$ in yellow. The support of $\brho$ is contained in the
  union of the light blue and dark blue rectangles and the hole in
  $\bk$-space is dark blue.  The plots in the lower left corners are
  the recovered magnitude-DFT coefficients in $W$ of the
  autocorrelation function, using the values found in~\eqref{eqn8.02}
  to ``fill the hole.'' The errors in the autocorrelation images are
  shown in the lower right corners.

 \begin{figure}[H]
  \centering
   \begin{subfigure}[H]{.45\textwidth}
        \centering
       \includegraphics[width=6.5cm]{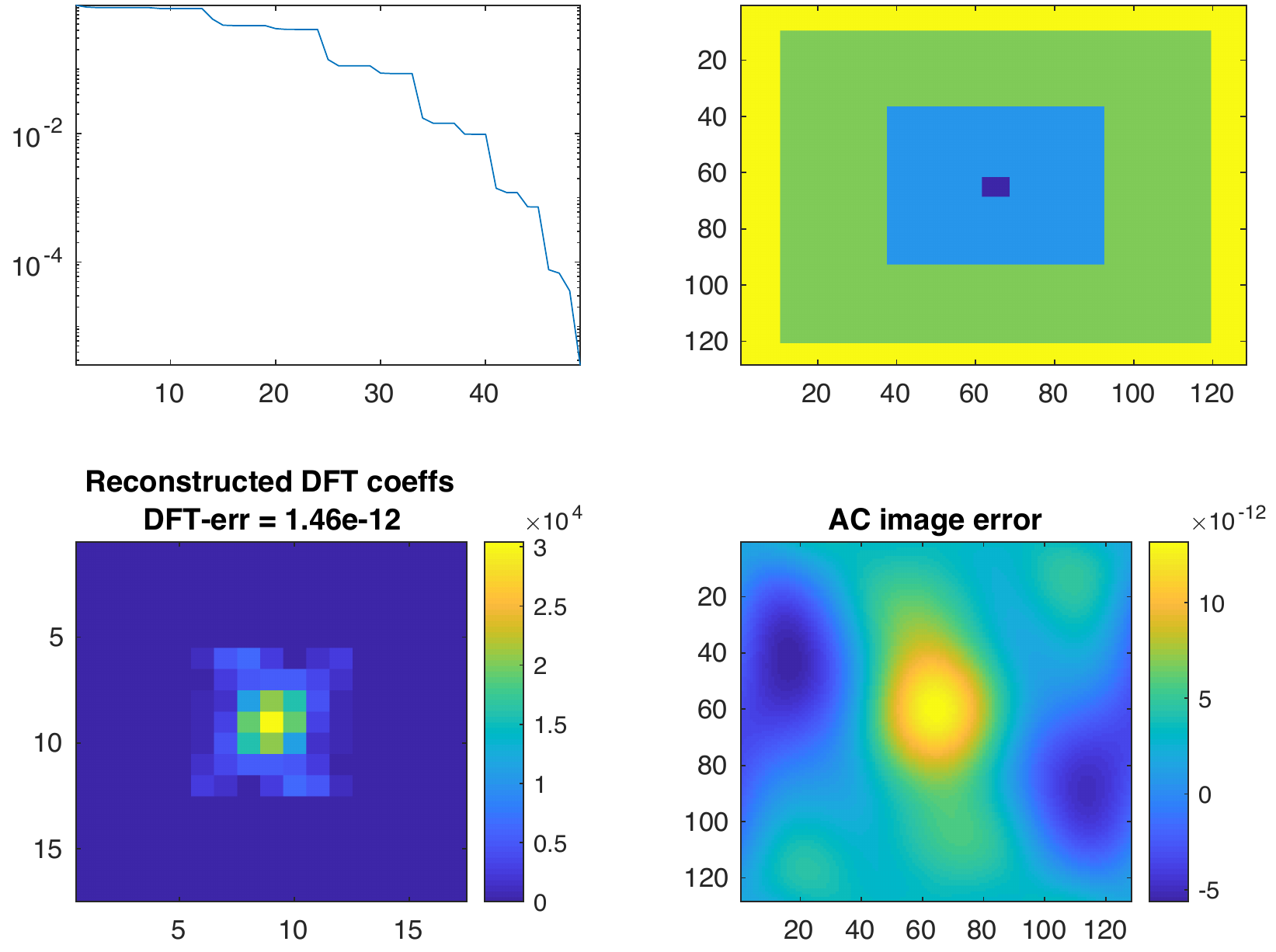}
        \caption{$W$ is a $7\times 7$ square.}
   \end{subfigure}\quad
       \begin{subfigure}[H]{.45\textwidth}
        \centering
       \includegraphics[width=6.5cm]{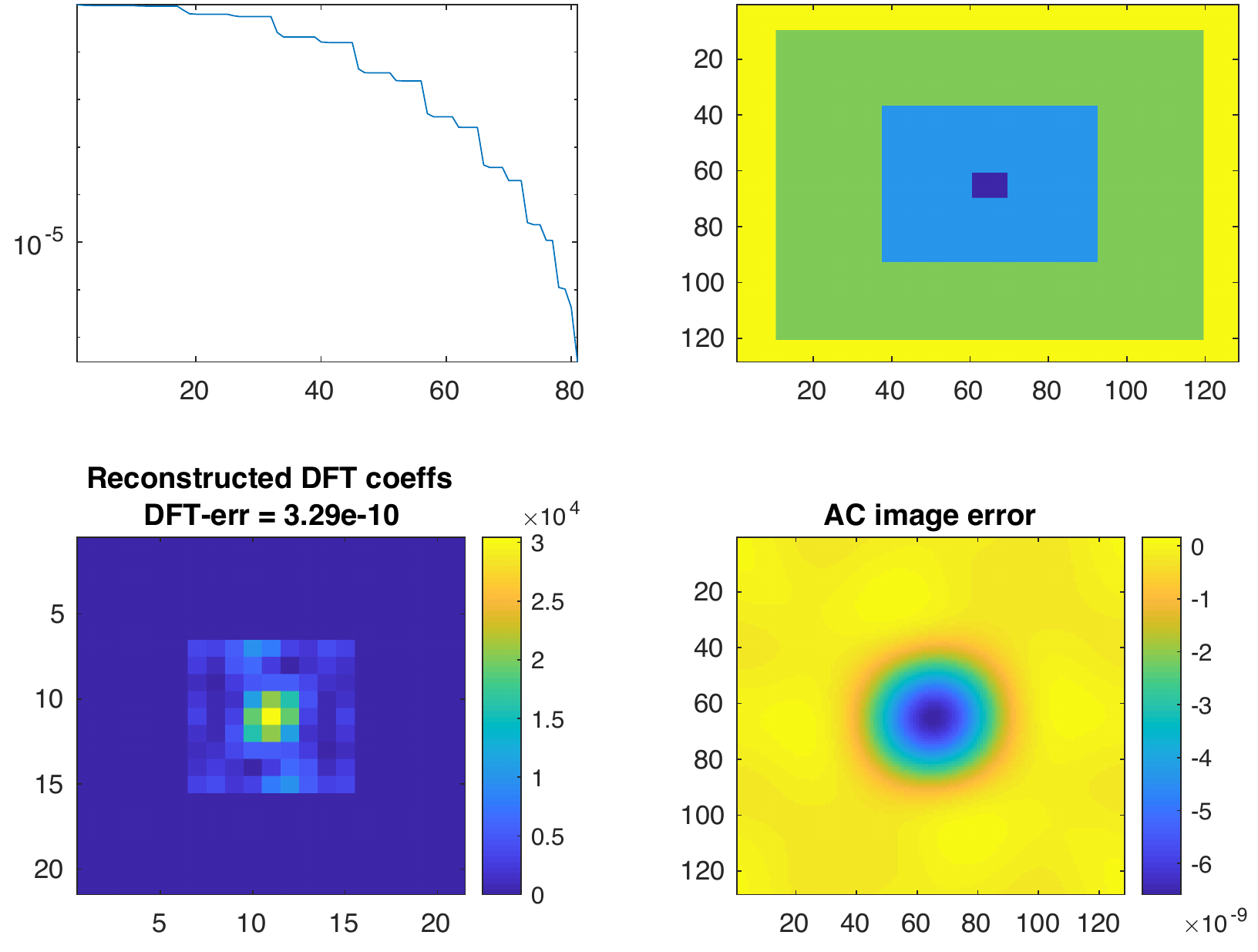}
        \caption{$W$ is a $9\times 9$ square.}
    \end{subfigure}\qquad
        \caption{Plots connected with the recovery of missing samples of the
          magnitude DFT data using double-oversampling.}\label{fig1}
  \end{figure}

    With triple oversampling we can recover the data with 11 digits of accuracy
    in a fairly large hole ($15\times 15$-hole in a $192\times 192$ grid), and the
    matrix $\cF^*_{R,W}$ has most of its singular values close to $1.$ With
    double oversampling the conditioning of the matrix $\cF^*_{R,W}$
    deteriorates more quickly.

     \begin{figure}[H]
  \centering
   \begin{subfigure}[H]{.45\textwidth}
        \centering
       \includegraphics[width=6.5cm]{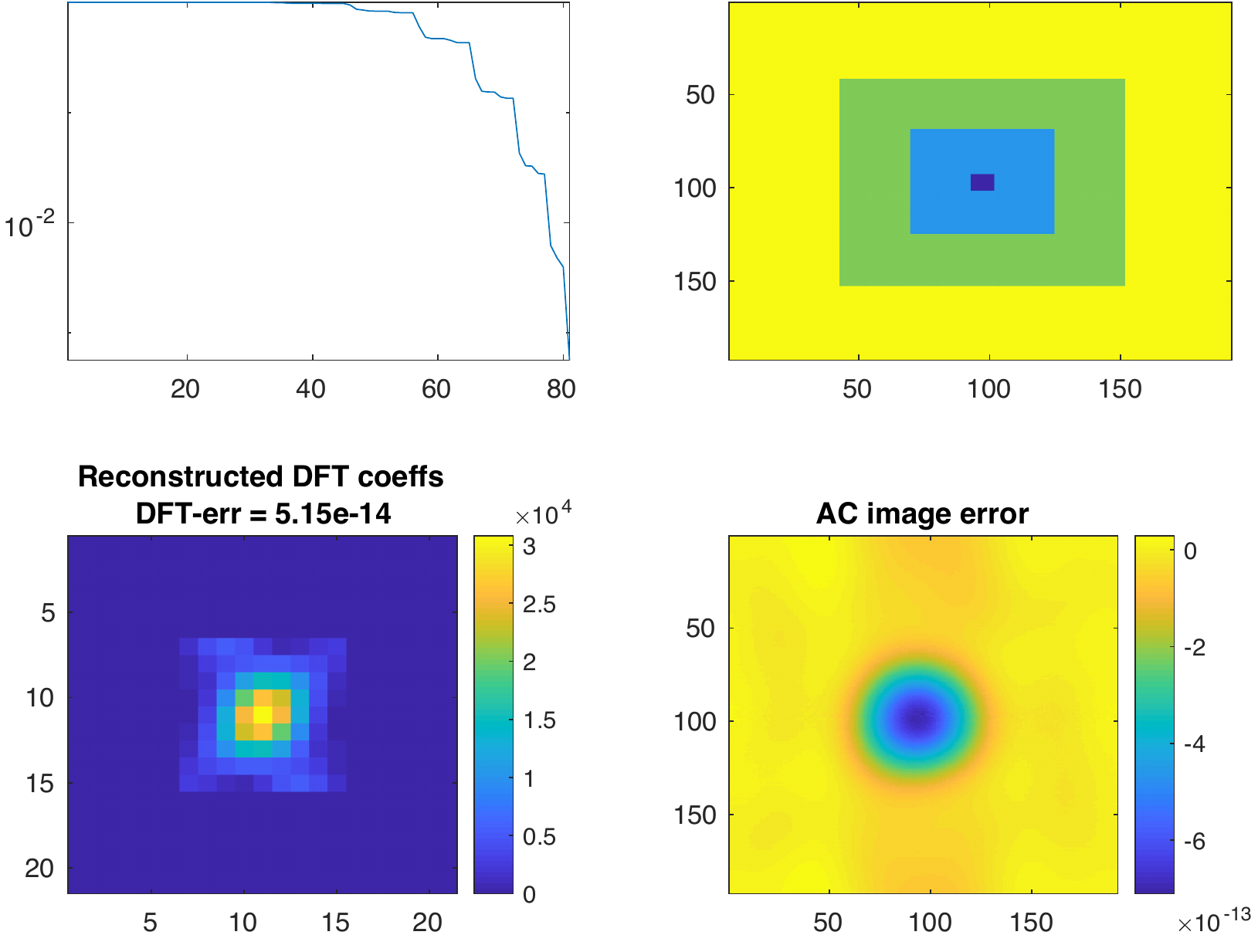}
        \caption{$W$ is a $9\times 9$ square.}
    \end{subfigure}\qquad
       \begin{subfigure}[H]{.45\textwidth}
        \centering
       \includegraphics[width=6.5cm]{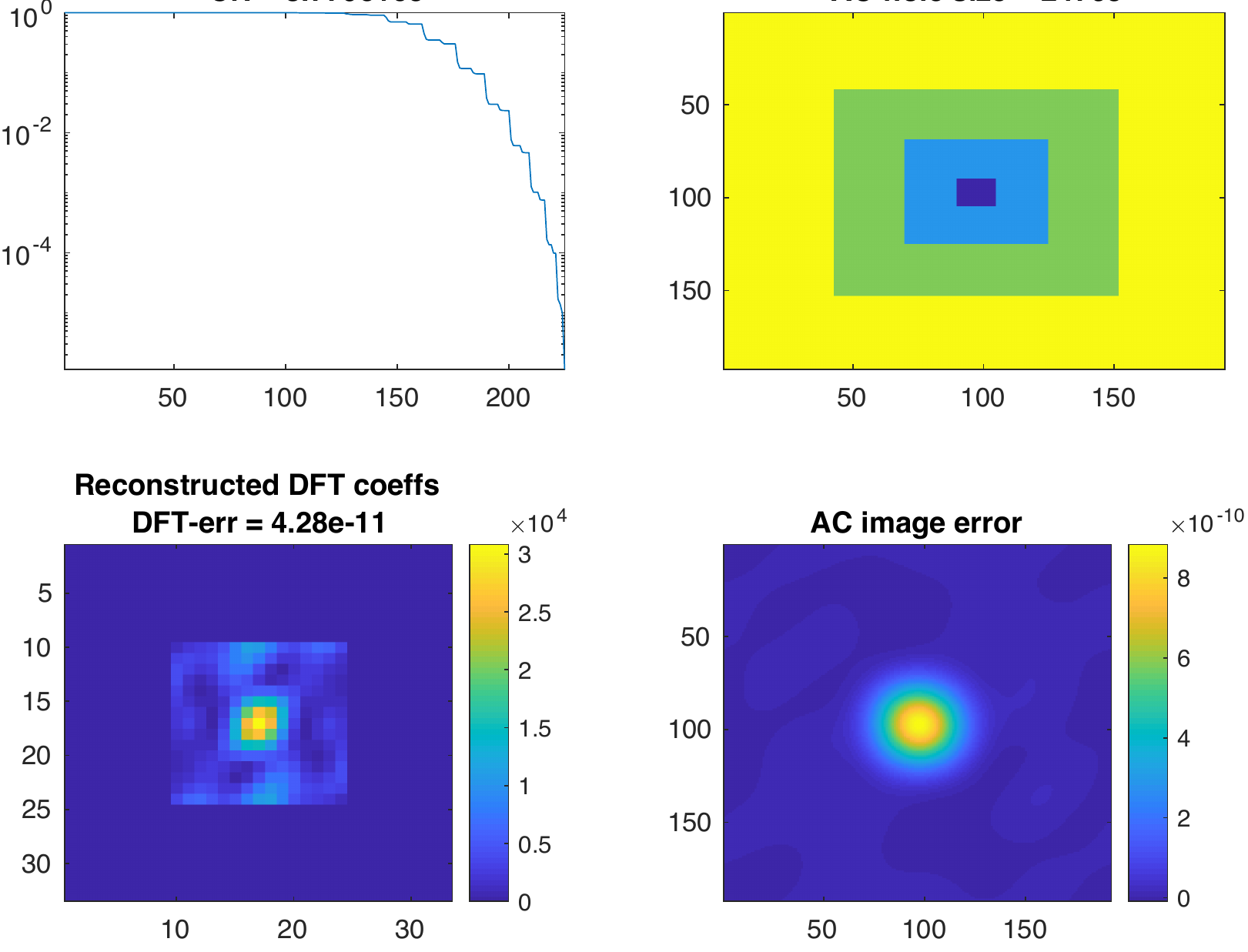}
        \caption{$W$ is a $15\times 15$ square.}
    \end{subfigure}\qquad
        \caption{Plots connected with the recovery of missing samples of the
          magnitude DFT data using triple-oversampling.}\label{fig2}
     \end{figure}
     
\end{example}

We now make a systematic study of the dependence of $\|\cR_{R,W}\|$ on
the various parameters that define this operator: $m,$ the degree of
oversampling, $N,$ the ``base'' number of samples, $k_{0}$ the
maximum spatial frequency not sampled. For these examples we fix
$\beta=1.5,$ so that the autocorrelation image is supported in
$[-1.5N:1.5N]^2.$ As predicted from the asymptotic
formula,~\eqref{eqn22.003}, the norm of $\cR_{R,W}$ increases
monotonically with $\beta.$
  
\begin{example}\label{exmpl3.01} 
  In these examples the images are indexed by $J=[1-mN:mN]^2,$ the
  image itself is supported in a proper subset of $[-N:N]^2,$ its
  autocorrelation image is supported in $[-1.5 N,1.5 N]^2,$ and
  samples of the magnitude DFT with indices in $[1-w:w-1]^2,$ are
  \emph{not} measured. As before, $w= \lfloor{1 + mk_0}\rfloor$.

  The asymptotic formula in~\eqref{eqn22.003} is expected to become
  increasingly accurate as $m,N$ grow. In fact, taking $N=128,$ and
  $m=3$ already leads to fairly good agreement with this estimate. To
  generate the tables below we fix $\beta =1.5$ and consider various
  values of $m,k_{0}$ for $N=64,128,256.$ Taking $m=3$ results in a
  large improvement over taking $m=2,$ but $m=4$ only provides a small
  improvement over $m=3.$

  \begin{table}[H]
      \footnotesize
    \makebox[\linewidth]{
\begin{tabular}{|l|l|l|l|l|l|}
\hline
 $N=64$       & $m=2$                & $m=3$               & $m=4$                       & Asymp. Val.       \\ \hline
$k_{0}=1$ & $337.8$              & $70.87$              & $45.4$                       & $20.06$           \\ \hline
$k_{0}=2$ & $1.73\times 10^5$    & $1.12\times 10^4$   & $5.42\times 10^3$  & $1.88\times 10^3$ \\ \hline
$k_{0}=3$ & $1.08\times 10^8$    & $1.99\times 10^6$   & $7.25\times 10^5$  & $1.89\times 10^5$ \\ \hline
$k_{0}=4$ & $8.61\times 10^{10}$ & $3.89\times 10^8$   & $1.05\times 10^8$   & $1.96\times 10^7$ \\ \hline
$k_{0}=5$ & $9.52\times 10^{13}$ & $8.44\times 10^{10}$ & $1.65\times 10^{10}$ & $2.06\times 10^9$ \\ \hline
\end{tabular}
}
\caption{Values  of $\|\cR_{R,W}\|$ for $\beta=1.5,$ $N=64$ and
  various choices of $m,k_{0}.$ The asymptotic values predicted
  by~\eqref{eqn22.003} are shown in the last column.}\label{tab1.0}
  \end{table}

  The following tables are generated with $N=128,$ and $N=256.$ The values in this 
   table that overlap with those in Table~\ref{tab1.0} are quite
   similar, with generally smaller values than for $N=64.$

   \begin{table}[H]
       \footnotesize
   \centering
     \begin{tabular}{|l|l|l|l|l|
       }
\hline
 $N=128$       & $m=2$                & $m=3$               & $m=4$              & Asymp. Val.       \\ \hline
$k_{0}=1$ & $361.03$              & $73.31$              & $46.7$              & $20.06$           \\ \hline
$k_{0}=2$ & $1.843\times 10^5$    & $1.17\times 10^4$   & $5.67\times 10^3$  & $1.88\times 10^3$ \\ \hline
$k_{0}=3$ & $1.05\times 10^8$    & $2.05\times 10^6$   & $7.52\times 10^5$  & $1.89\times 10^5$ \\ \hline
$k_{0}=4$ & $6.45\times 10^{10}$ & $3.77\times 10^8$   &
$1.04\times 10^8$   & $1.96\times 10^7$ \\ \hline\hline
$N=256$       & $m=2$                & $m=3$               & $m=4$              & Asymp. Val.       \\ \hline
$k_{0}=1$ & $373.17$              & $74.63$              & $47.4$              & $20.06$           \\ \hline
$k_{0}=2$ & $1.94\times 10^5$    & $1.21\times 10^4$   & $5.81\times 10^3$  & $1.88\times 10^3$ \\ \hline
$k_{0}=3$ & $1.09\times 10^8$    & $2.13\times 10^6$   & $7.76\times 10^5$  & $1.89\times 10^5$ \\ \hline
$k_{0}=4$ & $6.45\times 10^{10}$ & $3.88\times 10^8$   & $2.96\times 10^7$   & $1.96\times 10^7$ \\ \hline
     \end{tabular}
\caption{Values  of $\|\cR_{R,W}\|$ for $\beta=1.5,$ $N=128,256$ and
  various choices of $m,k_{0}.$ The asymptotic values predicted
  by~\eqref{eqn22.003} are shown in the last column.}
\end{table}
\end{example}
  
To close this section we consider the relationship in the errors of
the recovered
DFT magnitude data, versus that in the squared magnitude data. We
express the recovered autocorrelation magnitude data, $\{u^2_{\bk}:\:
\bk\in W\},$ as
\begin{equation}
  u_{\bk}^2=|\hrho_{\bk}|^2+\epsilon_{\bk},\text{ so that }
  \frac{|u_{\bk}^2-|\hrho_{\bk}|^2|}{|\hrho_{\bk}|^2}=
  \frac{\epsilon_{\bk}}{|\hrho_{\bk}|^2}.
\end{equation}
Clearly we have that
\begin{equation}
  u_{\bk}\approx |\hrho_{\bk}|+\frac{\epsilon_{\bk}}{2|\hrho_{\bk}|},
\end{equation}
and therefore
\begin{equation}
  \frac{||\hrho_{\bk}|-u_{\bk}|}{|\hrho_{\bk}|}\approx
  \frac{\epsilon_{\bk}}{2|\hrho_{\bk}|^2}.
\end{equation}
For $\bk$ near to zero, the magnitude DFT coefficients,
$|\hrho_{\bk}|,$ tend to be large, and therefore we can expect these
recovered values to have somewhat smaller relative errors than their
squared counterparts. This, however, does not mean that the relative
mean square error is smaller for $|\hrho_{\bk}|$ than for $|\hrho_{\bk}|^2.$
An example comparing these errors is shown in Figure~\ref{fig13}. This
resulted from filling a $13\times 13$-hole for a thrice oversampled
$64\times 64$-image. The data used here is noise-free.
\begin{figure}[H]
     \centering
       \includegraphics[width=6cm]{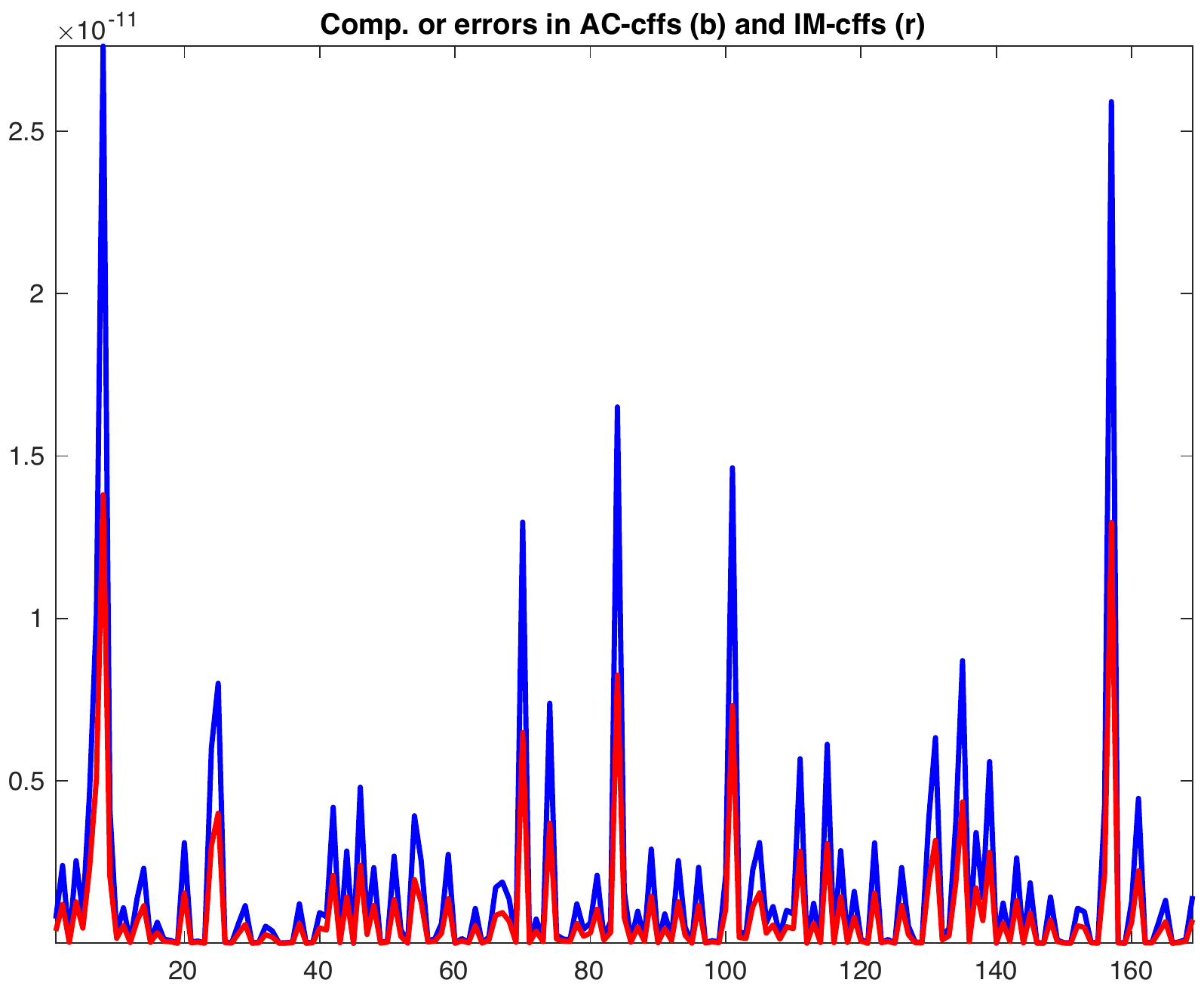}
       \caption{Relative errors in recovery of $\{|\hrho_{\bk}|^2\}$
         (blue curve)
         versus those for recovered values of $\{|\hrho_{\bk}|\}$ (red
       curve).}\label{fig13}
  \end{figure}

\section{The Effects of Noise}\label{noise.sec}
We now consider the effects of noise on the recovery process.  
Let $\bn\in\bbR^{W^c}$
represent the measurement error and noise. Then, instead of
solving~\eqref{lsqprob}, we actually need to solve the equation
\begin{equation}\label{eqn7n}
  \cF^*_{R,W}(\balpha_0+\bBeta)=-\cF^*_{R,W^c}(\ba+\bn).
\end{equation}
The relative effect of the noise introduced into $\balpha_0$ is then measured by
the ratio
\begin{equation}
  \frac{\|\bBeta\|}{\|\bn\|}= \frac{\|\cR_{R,W}\bn\|}{\|\bn\|}.
\end{equation}
The matrix $\cR_{R,W}$ has a representation of the form
\begin{equation}
  \cR_{R,W}=\sum_{j=1}^{|W|}\nu_j\bw_j\otimes\Bz_j^*,
\end{equation}
where $\{\bw_j:\: j=1,\dots,|W|\}$ is an orthonormal basis for the range and
$\{\Bz_j:\: j=1,\dots,|W|\},$ are pairwise orthonormal. For a vector $\bn$
\begin{equation}
 \|\cR_{R,W}\bn\|^2=\sum_{j=1}^{|W|}\nu_j^2|\langle\bn,\Bz_j\rangle|^2.
\end{equation}
The collection of vectors $\{\Bz_j\}$ can be augmented to give an
orthonormal basis, $\{\Bz_j:\: j=1,\dots,|W^c|\},$ for $\bbR^{W^c}.$
Hence for $\bn\in\bbR^{W^c},$ we have that
\begin{equation}
  \sum_{j=1}^{|W^c|}|\langle\bn,\Bz_j\rangle|^2=\|\bn\|^2.
\end{equation}

It is often reasonable to assume that the random variables
$\{|\langle\bn,\Bz_j\rangle|^2:\:j=1,\dots,|W^c|\}$ are independent and
identically distributed, and therefore the expected values satisfy:
\begin{equation}
  \bbE(|\langle\bn,\Bz_j\rangle|^2)=\frac{\|\bn\|^2}{|W^c|}.
\end{equation}
This would be the case for any additive, I.I.D.  noise process. 
In this case
\begin{equation}
  \bbE\left(\frac{\|\cR_{R,W}\bn\|^2}{\|\bn\|^2}\right)=\frac{1}{{|W^c|}}
  \left[\sum_{j=1}^{|W|}\nu_j^2\right].
\end{equation}
From the Cauchy-Schwarz inequality it follows that
\begin{equation}\label{eqn18.01}
  \bbE\left(\frac{\|\cR_{R,W}\bn\|}
           {\|\bn\|}\right)\leq\frac{1}{{\sqrt{|W^c|}}}
  \left[\sum_{j=1}^{|W|}\nu_j^2\right]^{\frac 12}.
\end{equation}
Even when the norm of $\cR_{R,W}$ is large, the quantity
appearing on the right hand side of~\eqref{eqn18.01} may turn out to be rather
modest.  This value gives a good estimate for the effect of noise on the
accuracy of the recovered values of the unmeasured DFT modulus data. The number
$\sqrt{|W^c|}\approx mN,$ (in $2d$), which shows that a potential  advantage of greater
oversampling is better noise suppression when recovering the unmeasured DFT
magnitude data.

As is well known, an important source of noise in CDI applications is 
Poisson noise that arises from the discreteness of X-ray photons.  This is
usually modeled as follows: if $|\hrho_{\bk}|^2$ is the ``true intensity'' of the
DFT coefficient in the $\bk$th pixel, then measurement $\ta^2_{\bk}$ is a sample
of a Poisson random variable with intensity $|\hrho_{\bk}|^2.$ The ``noise'' in
this pixel  is therefore given by
\begin{equation}
  n_{\bk}=\ta^2_{\bk}-|\hrho_{\bk}|^2.
\end{equation}
Clearly $\bbE( n_{\bk})=0,$ and $\bbE( n_{\bk}^2)=|\hrho_{\bk}|^2,$ and
therefore the SNR is $|\hrho_{\bk}|,$ which implies that the Poisson noise
process has a pixel dependent SNR.  As $\bbE(n_{\bk}^2)=\bbE(\ta^2_{\bk}),$
the noise is in some ways similar to the image itself. Indeed, the projection of
$\cF^*_{R,W^c}\bn$ into the range of $\cF^*_{R,W}$ tends be rather large.
 
Figure~\ref{fig11} shows histograms of the ratios,
$\frac{\|\cR_{R,W}\bn\|}{\|\bn\|},$ for different noise processes in a
triple oversampled example, where the condition number of $\cR_{R,W}$
is $1.093\times 10^{5}.$ These ratios are typically less than 400, for
uniform and Gaussian noise, and less than 2000, for Poisson noise. The
much smaller numbers in Gaussian and uniform cases are a reflection of
the fact that the orthogonal projection of $\cF^*_{R,W^c}\bn$ into the
range of $\cF^*_{R,W}$ tends to be quite small for $\bn$ a sample of
an additive I.I.D. noise process, as predicted
in~\eqref{eqn18.01}. As suggested by the discussion above, the
situation is rather different in the Poisson
case.

\begin{figure}[H]
  \centering
   \begin{subfigure}[H]{.3\textwidth}
        \centering
       \includegraphics[width=4cm]{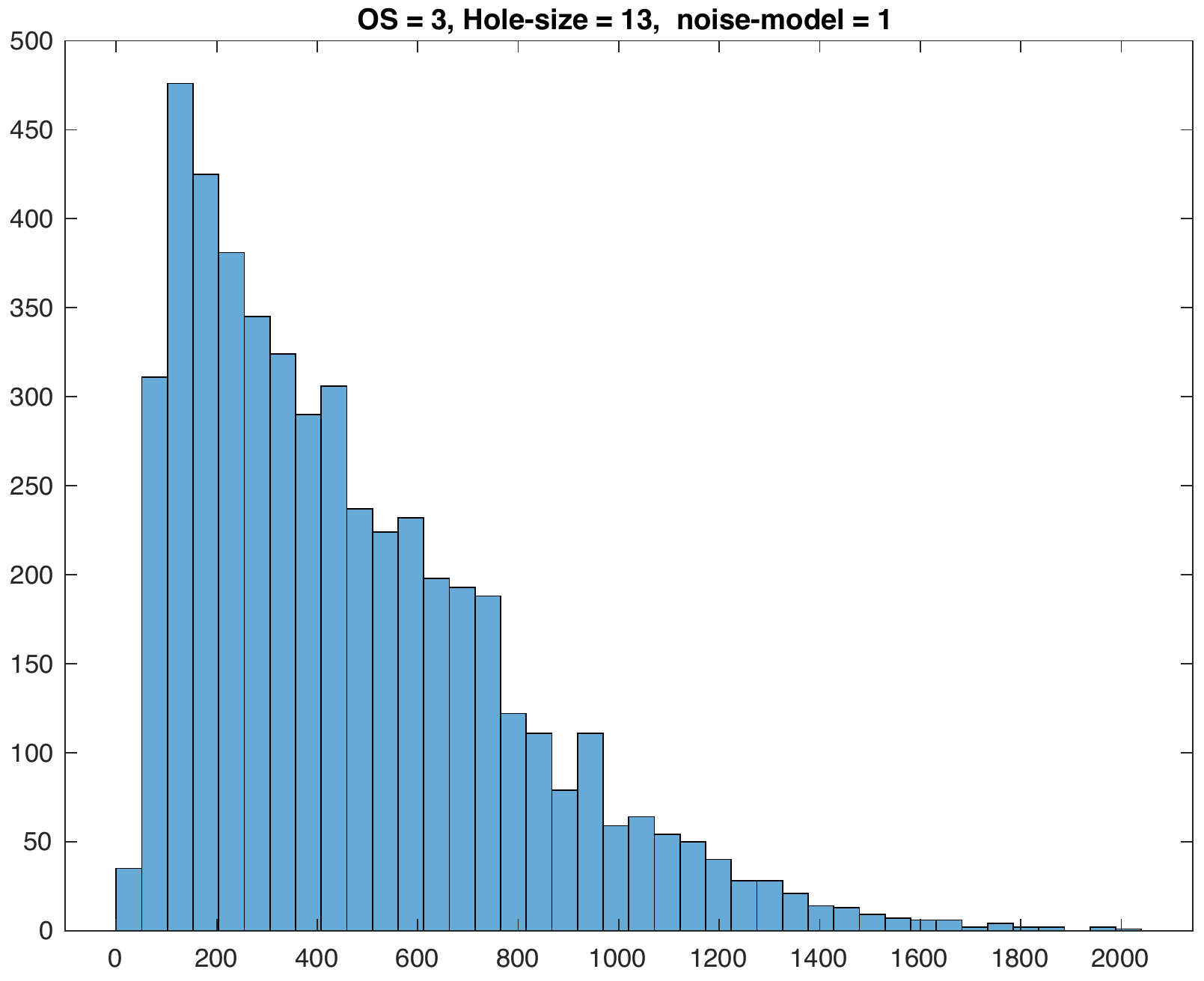}
        \caption{Uniform noise.}
   \end{subfigure}\quad
    \begin{subfigure}[H]{.3\textwidth}
        \centering
       \includegraphics[width=4cm]{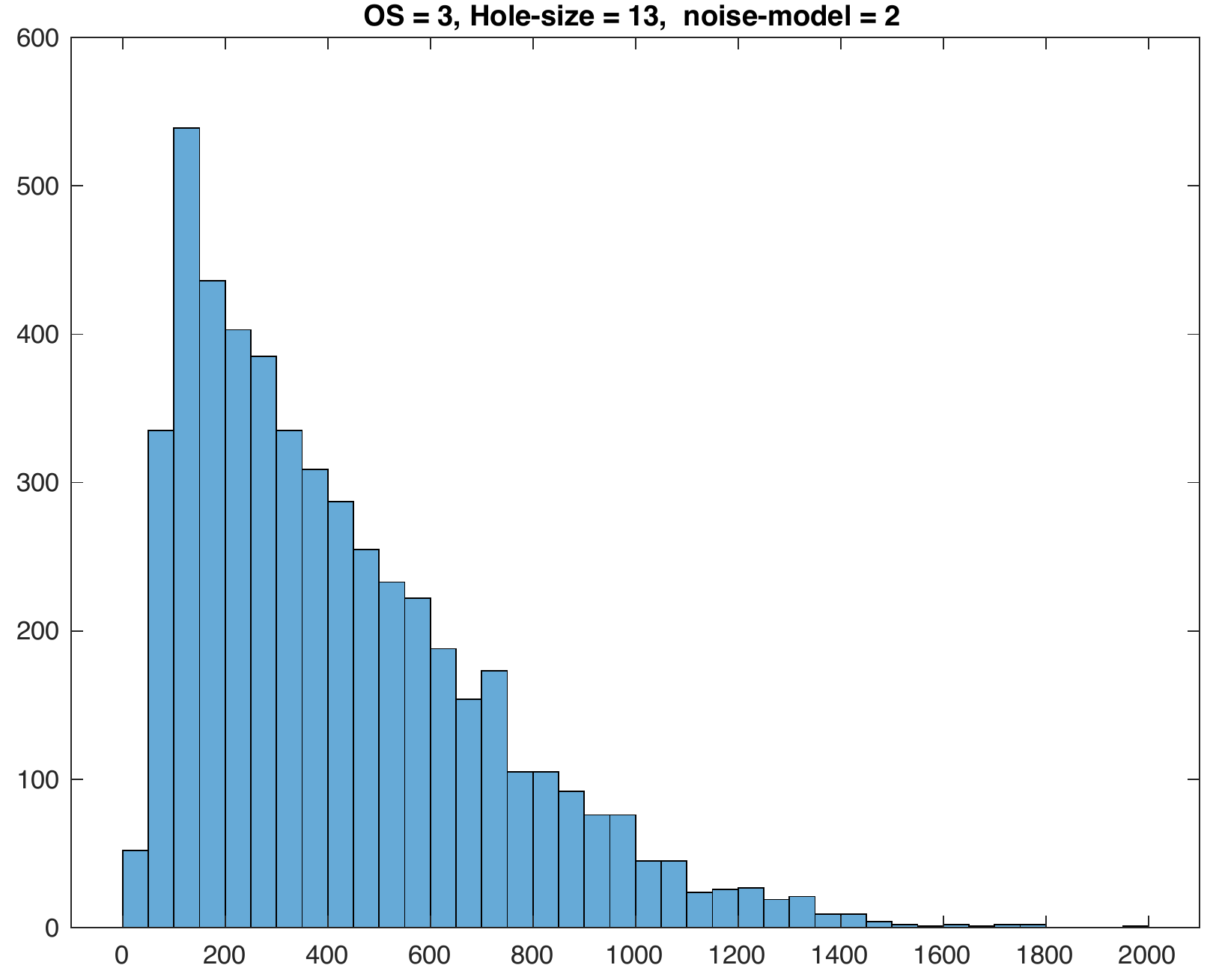}
        \caption{Gaussian noise.}
    \end{subfigure}\quad
     \begin{subfigure}[H]{.3\textwidth}
        \centering
       \includegraphics[width=4cm]{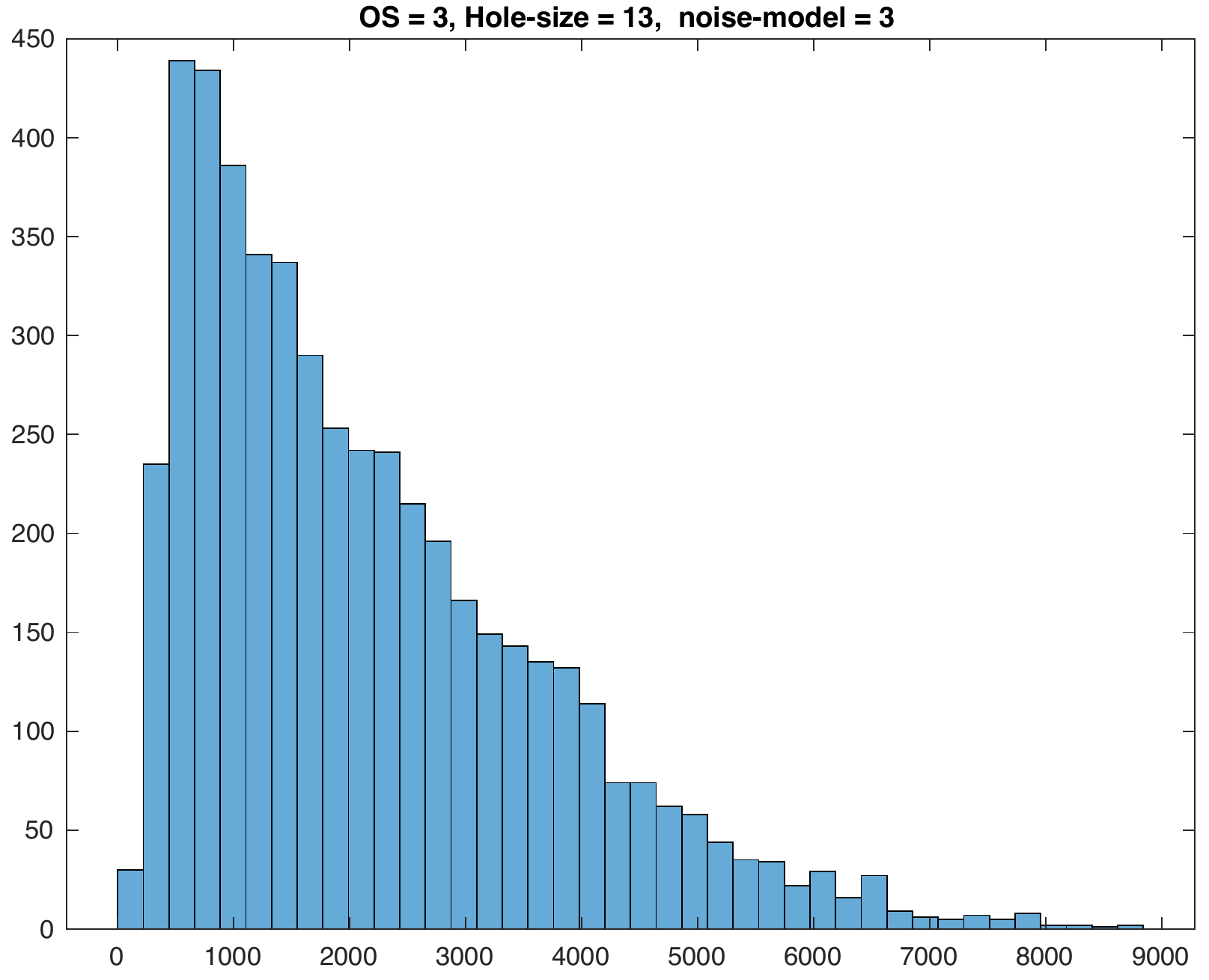}
        \caption{Poisson noise.}
    \end{subfigure}
        \caption{Histograms of the ratios
          $\|\cR_{R,W}\bn\|/\|\bn\|$ for 5000
          trials of uniform, Gaussian and Poisson noise. Here $w=7,
          m=3, N=64.$}\label{fig11}
  \end{figure}

\section{Hole Filling and Image Reconstruction}\label{sec6}
In this final section we consider how the hole-filling procedure
outlined above affects the outcome of image reconstruction using an
HIO-algorithm, see~\cite{Fienup1982, Bauschke:02}. This algorithm,
which iterates a map like that in~\eqref{eqn:alt-proj}, is currently
the basis for the best known, and most frequently used phase retrieval
method.  In the examples in this section we see that, for a certain
range of hole-sizes and in the absence of noise, the images obtained
by first filling in the unmeasured data using
equation~\eqref{eqn8.02}, and then using HIO are much better than
those obtained by simply using HIO. The picture is more complicated
when there is noise, with the results now depending on the character
of the noise and the SNR.  With noise, we find that it is often useful
to use some of the values recovered using equation~\eqref{eqn8.02},
and allow others to be filled in implicitly using HIO.

Suppose the data is of the form given by~\eqref{eqn4.01}, where $\balpha_{W^c}$ denotes the
measurements outside of the missing hole $W$.  Let  $P_A$  denote the
projection operator onto the nearest point in some set $A.$ 
We set $B_S= \{\brho : P_{S^c}(\brho) =
0\}$ and $\TA_{\ba}= \{\brho : \abs{P_{W^c}\circ\mathcal{F}(\brho)}^2 =
\balpha_{W^c}\}$. HIO and related algorithms provide an update of the form:
\begin{equation} \label{eqn:alt-proj}
  \brho^{(k+1)} = \brho^{(k)}+P_{\TA_{\ba}}\circ R_{B_s}(\brho^{(k)})-
    P_{B_s}(\brho^{(k)}),
\end{equation}
where $R_{B_s}(\brho)=2P_{B_s}(\brho)-\brho.$

Note that \eqref{eqn:alt-proj} operates agnostically in regards to the
missing data inside $W$, for every missing data value one less
constraint equation is imposed. Thus, conceivably filling in the
missing data in $W$ before applying HIO (or any such phase retrieval
algorithm) could improve the quality of the reconstructed image.

Extensive numerical simulations indeed confirm this to be true. We fix
a test image and the set of corresponding squared DFT magnitude
measurements, $\ba_{W^c}^2,$ with low frequencies removed that belong
to a square, $W,$ of size $(2w-1)\times(2w-1)$ centered on $(0,0).$ 
Here, we use triple oversampling so that $w = \lfloor*{ 1+  3 k_0}\rfloor$.
We
then compare the following two recovery procedures: (i) HIO is
directly applied to the ``measured'' data $\ba^2_{W^c}$, and (ii) the
missing data in $W$ is first filled in using the recovery operator,
and then HIO is applied to the full data set $\ba^2$ (henceforth
referred to as the ``Fill+HIO'' algorithm). It is observed that
Fill+HIO produces superior image reconstruction for values of $k_0$ for
which the linear system, given by~\eqref{lsqprob}, can be solved
accurately.  Fill+HIO provides improved recovery up to $w=15$ 
($k_0 \approx 5$),
whereas HIO alone fails after $w=6$ ($k_0 \approx 2$).
Typical comparative results on simulated CDI
data are shown in Figs~\ref{fig:clean-compar} and
~\ref{data-table}.

Practical approaches for the phase retrieval problem in the
presence of noisy data typically involve numerical
optimization~\cite{Shi:18,Barmherzig2020} and data-driven methods,
see~\cite{pmlr-v80-metzler18a}, topics that are outside the scope
of this paper and which we do not pursue further. However,
we do provide some general remarks, and guidelines for applying
the Fill+HIO algorithm to problems with noisy data.

\begin{figure}[H] 
  \centering
 \includegraphics[width=15cm]{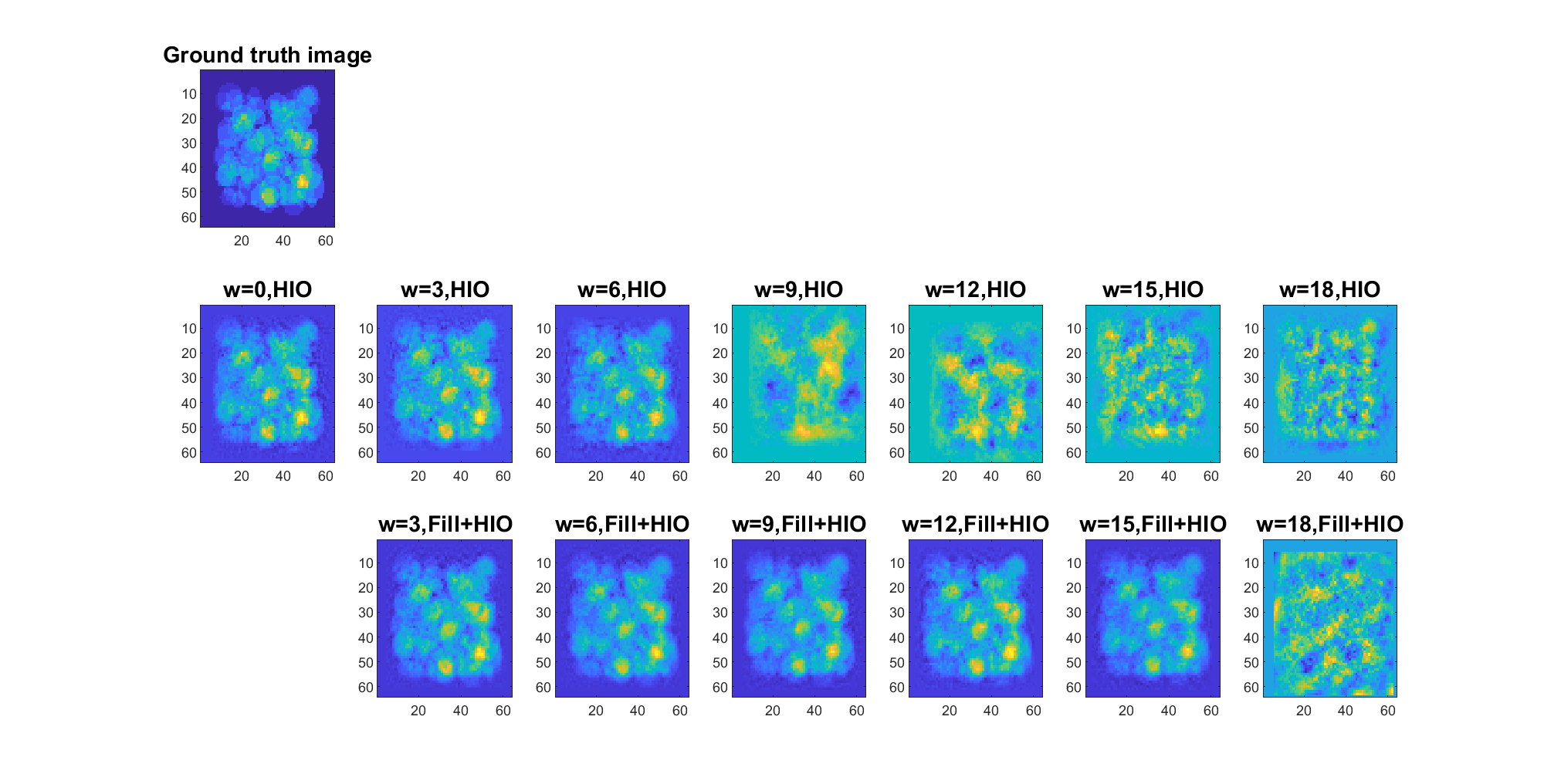}
\caption{Image reconstruction via HIO alone, in the middle row, versus
  phase retrieval using hole-filling followed by HIO (``Fill+HIO'') in
  the bottom row. The image is of size $64\times 64$, and data is of
  size $192 \times 192,$ so $m=3,$ and the $(2w-1)\times (2w-1)$
  frequencies centered on $(0,0)$
  zeroed-out. }\label{fig:clean-compar}
\end{figure}

When used with real measurements, the filled-in data values obtained
via~\eqref{eqn8.02} are necessarily contaminated by noise. Thus,
there arises a tradeoff between ignoring the missing data and
first recovering estimates for these values which contain errors. It is
observed from numerical simulations that, with noisy data, the best
image reconstruction is achieved by utilizing a subset of the data
found using~\eqref{eqn8.02}, and allowing HIO to recover the remaining
coefficients.

A natural procedure for determining the best subset to choose is to run multiple
trials of the Fill+HIO procedure where, for each trial, the amount of recovered
data that is used is incrementally increased. While, in practice, the true
smallest error achieved throughout such trials is unknown (since the
ground-truth image is unknown), an empirically successful proxy is to consider
the \textit{data error} for each trial; if $\brho$ is the approximate
reconstruction satisfying the support condition, then the data error is:
\begin{equation} \label{eqn:data-err}
\frac{\norm{\abs{\widehat{\brho}}^2-\abs{\widehat{\brho_0}}^2}{2}}{\norm{\abs{\widehat{\brho_0}}^2}{2}}.
\end{equation}
For our experiments we choose the partial filling that minimizes this quantity.

\begin{figure}[H]
  \centering
       \includegraphics[width=12cm]{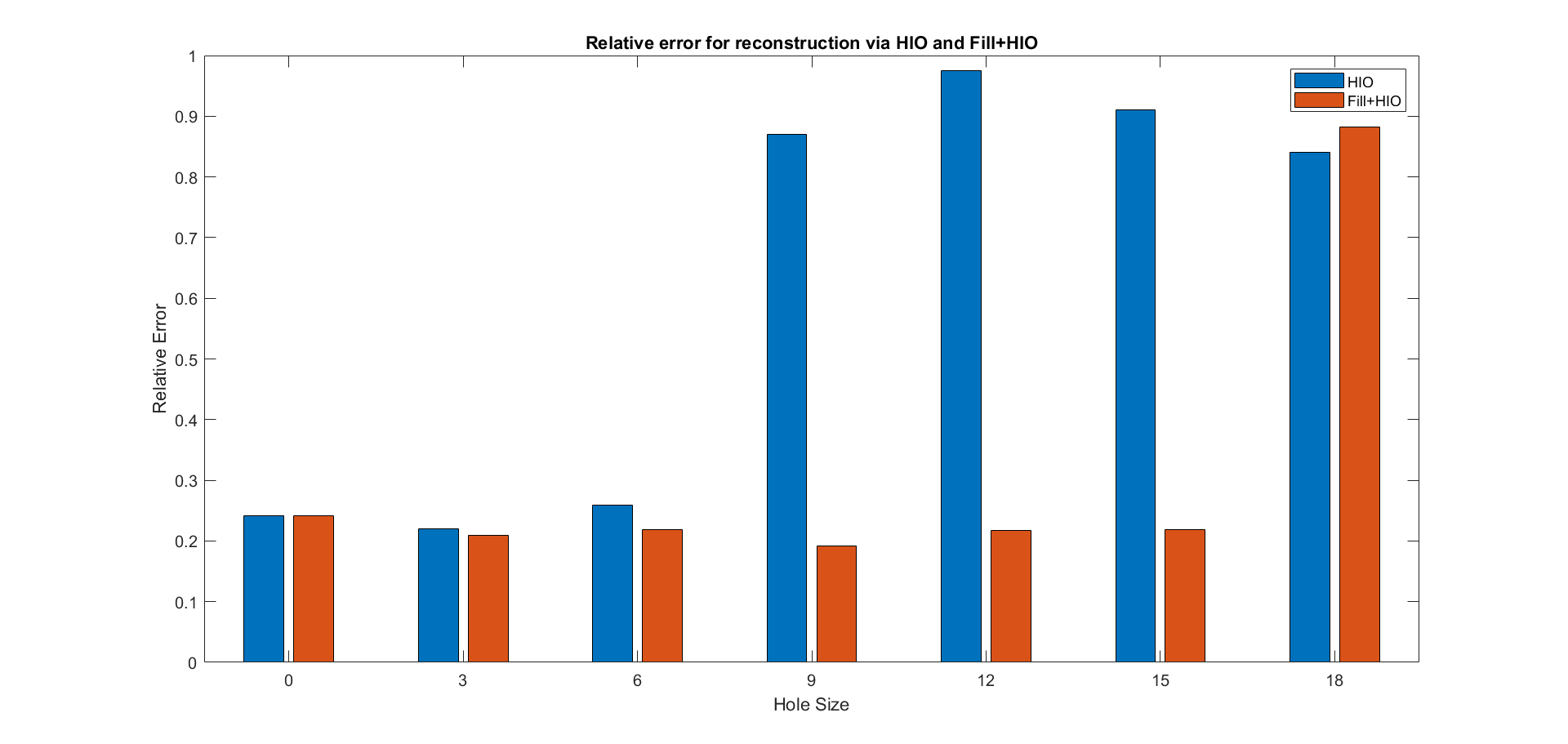}
        \caption{
Relative errors $\frac{\norm{\brho-\brho_0}{2}}{\norm{\brho_0}{2}}$
for the ground truth image $\brho_0$ shown in Fig. 1 and the recovered images
$\brho$ using the HIO (blue) and Fill+HIO (orange) algorithms, respectively. Data has the
$(2w-1)\times(2w-1)$ square of lowest frequencies zeroed-out.
}
\label{data-table}
\end{figure}


We concentrate on the case of data corrupted by Poisson noise,
such as typically occurs in CDI experiments. The discussion at the end
of Section~\ref{s.rec_alg} clearly indicates that the largest
amplification of noise occurs in the recovery of the lowest-frequency
values. Thus, a natural search strategy for partially filling a
rectangle of missing data with recovered values is to work from the
outer boundary of $W$ inward,  considering annular  regions,
which restore the mid-range of missing frequencies.

We apply this procedure to simulated CDI data, corresponding to the
setup in Figs.~\ref{fig:clean-compar} and~\ref{data-table},
when $w=5, m=3$, that is corrupted by Poisson noise with a
signal-to-noise ratio of 1000. Over 1000 trials, we observe that the
distribution of the recovery error is noticeably improved by restoring
some of the missing data before running HIO. This is illustrated in
the histograms shown in Figure~\ref{fig:noisy-comp}.

\begin{figure}[H] 
\centering
        \includegraphics[width=0.8\textwidth]{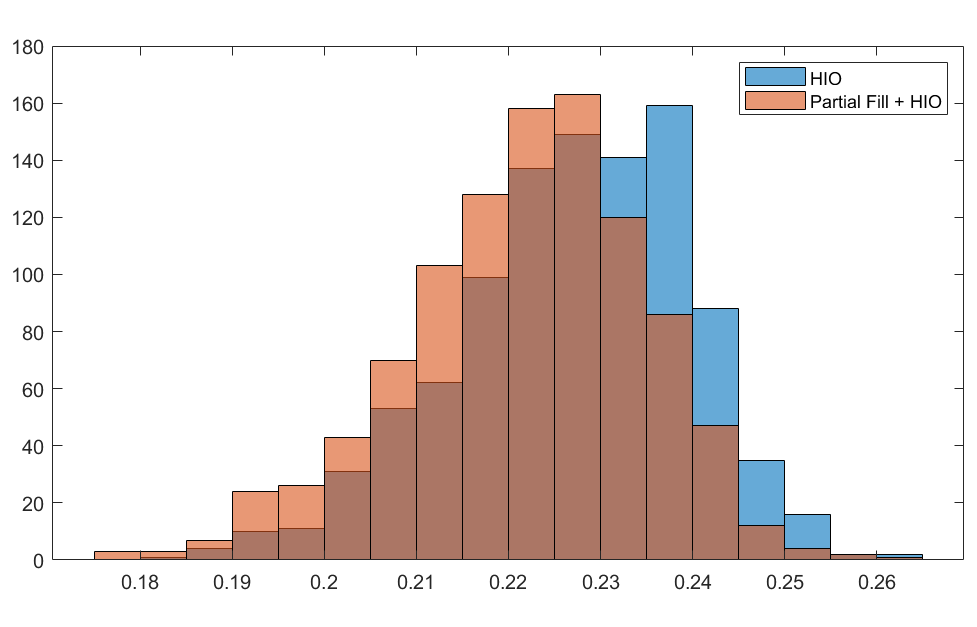}
    \caption{Histograms of relative error values generated from 1000 noisy
      instances of the simulated CDI setup shown in
      Figure~\ref{fig:clean-compar}, using the HIO (blue) and Partial
      Fill+HIO (orange)
      algorithms, respectively. Partial Fill+HIO significantly improves the
      error distribution.}\label{fig:noisy-comp}
\end{figure}

\section{Conclusions}\label{concl}

In this paper, we have investigated the problem of recovering the 
unmeasured data within the beamstop in CDI imaging.
Rather than including the full complex Fourier transform at these missing
locations as part of a global inverse problem, we have shown that the 
modulus Fourier data within the beamstop can itself be recovered, as the 
solution to a linear least squares problem. Algorithms for phase 
retrieval can then be used in a second step on this ``filled in" data set.
The power of this approach is illustrated in Fig.  \ref{data-table}.

We also analyzed under what conditions this method of recovery is
likely to be successful. If $W$ and $S_{AC}$ are rectangular subsets
of $J,$ then the answer hinges on the value of the dimensionless
parameter $\beta k_0$, here $\beta>1$ is determined by the support of
the autocorrelation function $\rho\star\rho,$ and $2k_0$ is the hole
width. In this analysis, we assume the object of interest has spatial
dimensions normalized to unit length.

Our analysis provides a generalization of Hayes' theorem
\cite{Hayes1982} to the case of phase retrieval with missing data.
Theorem~\ref{thm0}, shows that, very often, the missing data, hidden
by the beamstop, can be \emph{uniquely} recovered from the measured
magnitude data. Hayes' theorem then applies directly to the completed
data set to show that the solution to the phase retrieval problem is
again generically unique.

This method for recovering the unmeasured magnitude data should be
applicable to many classical phase retrieval problems. We are
currently investigating the extension of our results to other
X-ray imaging modalities.

\appendix
\section{Appendix}\label{A1}
In this appendix we derive an asymptotic bound for $\mu_0(R,W,d),$
assuming that $W$ and $S_{AC}=R^c$ are rectangular subsets. This bound
becomes more accurate as $m,N$ tend to infinity and the product $\beta
k_{0}$ grows. A similar question is addressed in Barnett's recent
paper~\cite{BarnettDFT2020} on the conditioning of sub-blocks of the
DFT matrix. An upper bound on $\mu_0(R,W,1),$ follows from the
estimates in Barnett's paper.

In~\eqref{eqn10.01} we show that
\begin{equation}\label{eqn41.003}
  \mu_0(S_{AC},W,d)=1-\max_{\balpha_W\neq
    0}\frac{\|\cF^*_{S_{AC},W}\balpha_W\|^2}
  {\|\balpha_W\|^2}.
\end{equation}
It should first be noted that as
$\cF^*\balpha=[\cF(\overline{\balpha})]^*,$ replacing $\cF^*$ with
$\cF$ does not change the value of the maximum in this formula.  The
key consequence of this formula is that the smallest singular value
of $\cF^*_{R_,W}$ is determined by the largest singular value of
$\cF^*_{S_{AC},W}.$ Because $S_{AC}\subset J$ is a rectangular set,
and the $d$-dimensional DFT is a tensor product of 1-dimensional DFTs,
this allows the determination of these singular values as products of
singular values that arise in the 1-dimensional case.

With this in mind, we let $\balpha$ be a sequence of length $2mN$
supported in $[1-w:w-1],$ and $f,$  a  real valued function
supported in $[-k_{0},k_{0}],$ where $w = \lfloor{1 + mk_0}\rfloor,$
with
  $f\left(\frac{j}{m}\right)=\sqrt{\frac{m}{2N}}\alpha_j.$ The  discrete
Fourier  transform of $\balpha$ is
\begin{equation}\label{eqn39.002}
  \hat{\alpha}_k=\frac{1}{m}\sum_{j=-w}^{w}f\left(\frac{j}{m}\right) e^{-\frac{2\pi i jk}{2mN}}\approx
\hf\left(\frac{k}{2N}\right).
\end{equation}
With these approximations it follows
that
\begin{equation}
  \begin{split}
  \sum_{k=-\beta N}^{\beta N}|\hat{\alpha}_k|^2&\approx\sum_{k=-\beta
    N}^{\beta N}\left|\hf\left(\frac{k}{2N}\right)\right|^2,\\
  \sum_{j=1-w}^{w-1}|\alpha_j|^2&\approx
2N \int_{-k_{0}}^{k_{0}}|f(x)|^2dx,
 \end{split}
\end{equation}
and therefore the ratio, whose maximum defines $\mu_0([-\beta N:\beta
  N],[1-w:w-1],1),$ is approximated by:
\begin{equation}\label{eqn41.002}
  \frac{ \sum_{k=-\beta N}^{\beta N}|\hat{\alpha}_k|^2}
       {\sum_{j=1-w}^{w-1}|\alpha_j|^2}\approx
  \frac{\int_{-\frac{\beta}{2}}^{\frac{\beta}{2}}|\hf(y)|^2dy}{ \int_{-k_{0}}^{k_{0}}|f(x)|^2dx}.
\end{equation}
As noted above, $\balpha$ is a sequence supported in $[1-w:w-1].$ The ratio of the
sums on the left hand side converge 
to the ratio of integrals on the right hand
side as $N,m\to\infty$.

We define
\begin{equation}
  \lambda_0(k_{0},\beta,1)=\max_{f\in \cA^2_{k_{0}}\setminus
    \{0\}}
  \left[ \frac{\int_{-\frac{\beta}{2}}^{\frac{\beta}{2}}|\hf(y)|^2dy}{ \int_{-k_{0}}^{k_{0}}|f(x)|^2dx}\right],
\end{equation}
where $\cA^2_{k_{0}}$ consists of functions in $L^2(\bbR)$ supported
in $[-k_{0},k_{0}].$ The calculations above show that, at least
asymptotically, as $m,N$ grow,
\begin{equation}\label{eqn43.002}
  \mu_0(S_{AC},W,1)\approx 1-\lambda_0(k_{0},\beta,1).
\end{equation}
The quantity on the right hand side of~\eqref{eqn43.002} has been
intensively studied in the literature on prolate spheroidal functions,
see~\cite{Fuchs1964,Slepian1978}; adapting the result of Theorem 1
from~\cite{Fuchs1964} we obtain:
\begin{equation}\label{eqn44.002}
    \mu_0(S_{AC},W,12)
  \sim 4\pi\sqrt{\beta  k_{0}}e^{-2\pi \beta k_{0}}.
\end{equation}
The asymptotic evaluation on the right hand side of~\eqref{eqn44.002} is in the
limit $\beta k_{0}\to\infty.$

In $d$ dimensions,  suppose that the support of the autocorrelation image is
contained in the cuboid $[-\frac{\beta}{2},\frac{\beta}{2}]^d$ and
$W=[1-w:w-1]^d.$ Let $\lambda_0(k_0,\beta,d)$ be the $d$--dimensional analogue
of $\lambda_0(k_0,\beta,1);$ the extremizer defining $\lambda_0(k_0,\beta,d),$ is just
the $d$--fold tensor product of the $1d$--extremizer.  Hence we see that
\begin{equation}
  \begin{split}
  \lambda_0(k_0,\beta,d)=  \lambda_0(k_0,\beta,1)^d
  &\sim (1-4\pi\sqrt{\beta k_{0}}e^{-2\pi \beta k_{0}})^d\\
  &\approx 1-4d\pi\sqrt{\beta k_{0}}e^{-2\pi \beta k_{0}},
  \end{split}
\end{equation}
and therefore
\begin{equation}
   \mu_0(S_{AC},W,d)\sim 4d\pi \sqrt{ \beta k_{0}}e^{-2\pi \beta k_{0}} .
\end{equation}
The norm of $\cR_{R,W}$ might be smaller than $[\mu_0(R,W,d)]^{-\frac
  12}$, as $\cF^*_{R,W^c}$ has norm less than 1. In fact, in our
applications, the norm of $\cF^*_{R,W^c}$ is very close to 1. 
Hence the norm of the recovery operator is given
asymptotically by the quantity
\begin{equation}\label{eqn48.004}
  [\mu_0(S_{AC},W,d)]^{-\frac 12}\sim  \frac{e^{\pi \beta k_{0}} }{\sqrt{4d\pi}[\beta k_{0}]^{\frac 14}}.
\end{equation}
Since the minimum singular value of $\cR_{R,W}$ is very close to 1,
this is also an asymptotic estimate for the condition number.

\begin{figure}[H]
  \centering
   \begin{subfigure}[H]{.3\textwidth}
        \centering
       \includegraphics[width=4cm]{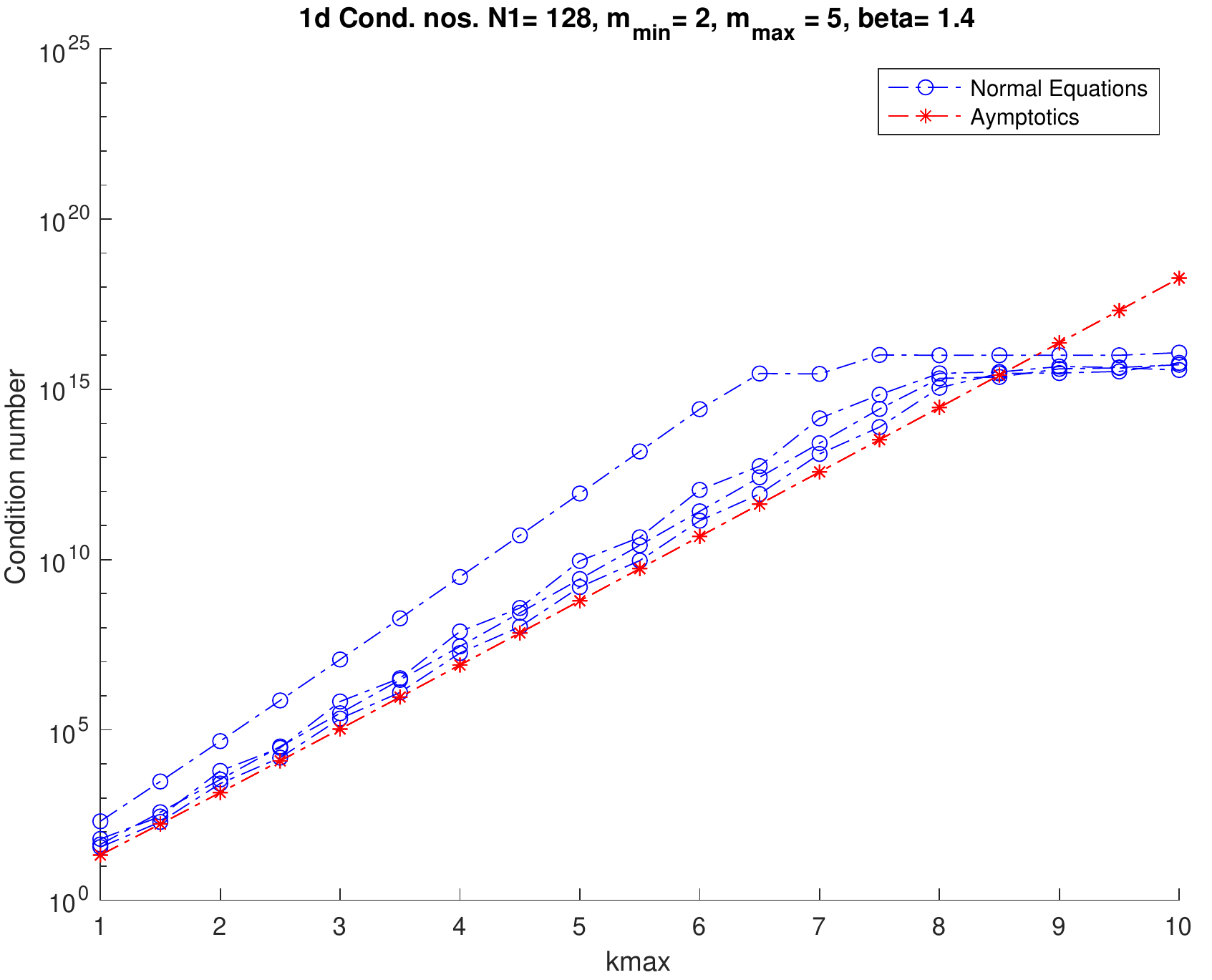}
        \caption{$\beta = 1.4$}
   \end{subfigure}\quad
    \begin{subfigure}[H]{.3\textwidth}
        \centering
       \includegraphics[width=4cm]{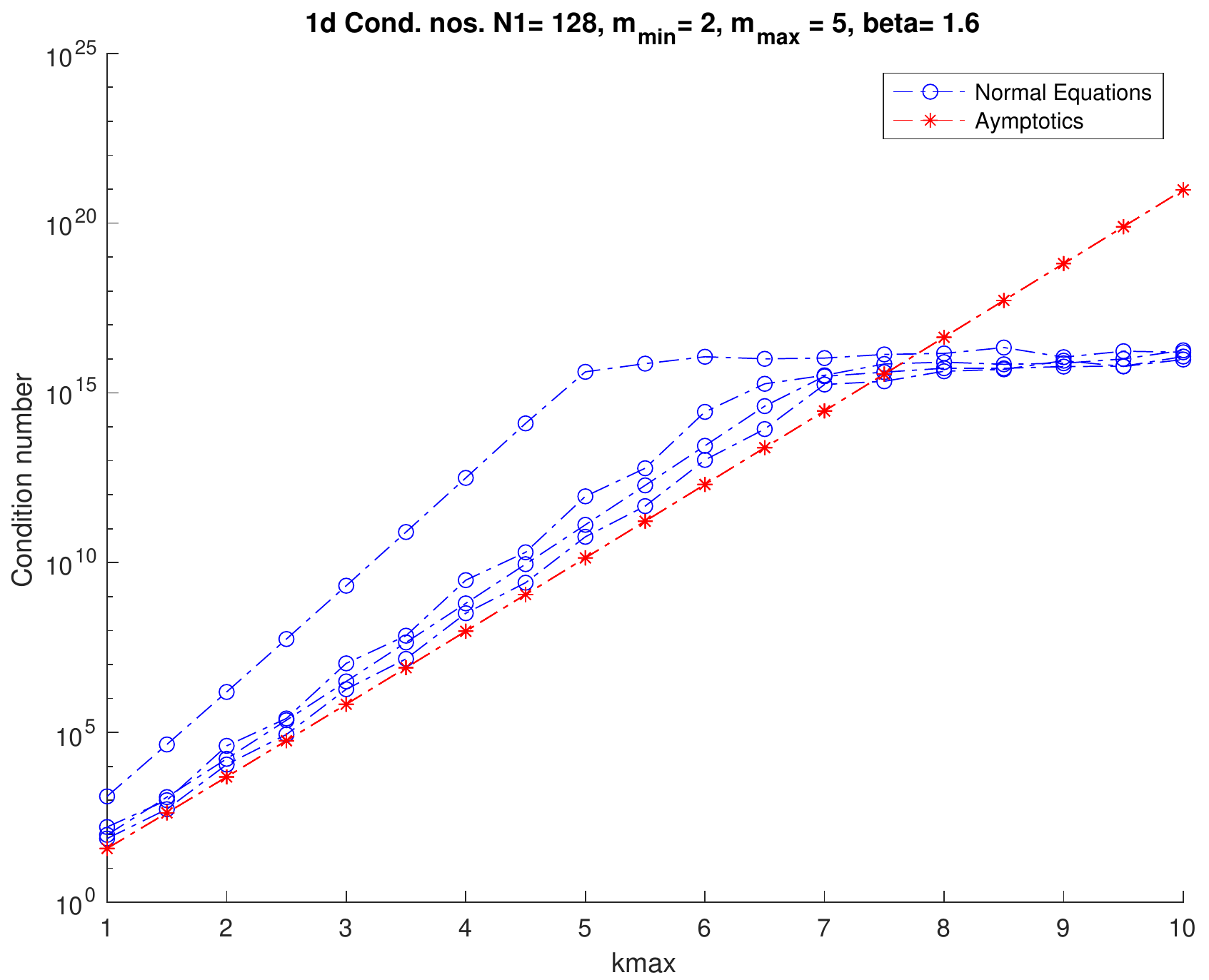}
        \caption{$\beta = 1.6$}
    \end{subfigure}\quad
     \begin{subfigure}[H]{.3\textwidth}
        \centering
       \includegraphics[width=4cm]{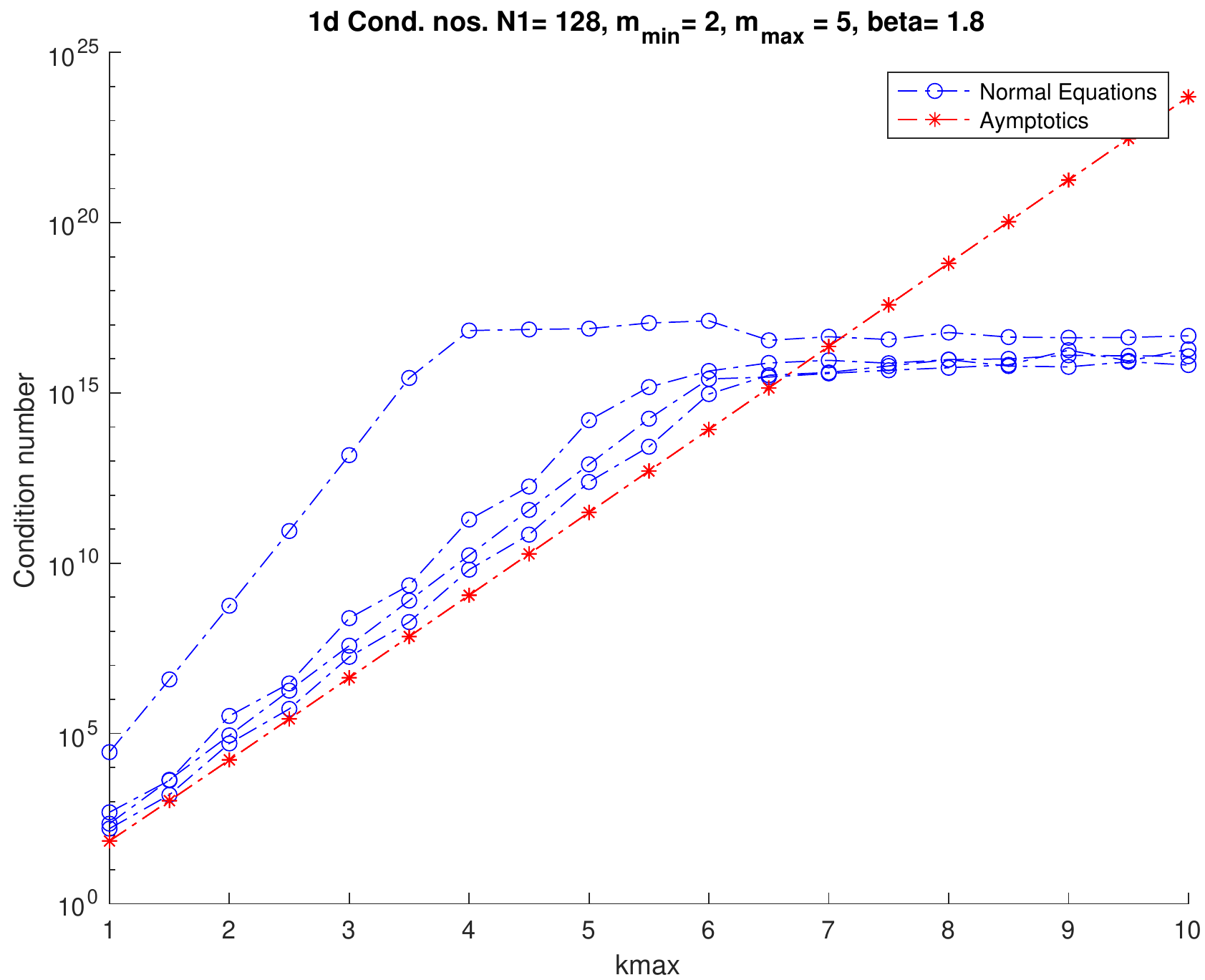}
        \caption{$\beta = 1.8$}
    \end{subfigure}
        \caption{Plots of $[\mu_0(S_{AC},W,1)]^{-\frac 12},$ for
          $m=2,3,4,5,$ (in blue) along with a plot of the asymptotic
          formula (in red). In these plots $k_{0}$ ranges from 1 to
          10, and $y$-axis goes from $10^0$ to $10^{25}.$}\label{fig13.04}
  \end{figure}

In Figure~\ref{fig13.04}[a,b,c] we show values of
$[\mu_0(S_{AC},W,1)]^{-\frac 12},$ for $\beta=1.4, 1.6, 1.8.$ In each
plot, there are 4 blue curves corresponding to $m=2,3,4,5,$
along with the predictions (in red) made
by~\eqref{eqn48.004}, with $d=1.$  The $x$-axis is $k_{0},$ which ranges from
$1$ to $10.$ As $m$ increases, the blue curves get closer to the plot
of the asymptotic formula. As long as there is sufficient accuracy in
the double precision calculation, the asymptotic formula is close to
exact calculation by the time $m=3,$ and is a lower bound throughout
this range of parameters. Once the condition number reaches $\sim
10^{15},$ the calculation of $[\mu_0(S_{AC},W,1)]^{-\frac 12}$
saturates and is no longer meaningful. In these computations $N=128,$
and increasing it does not significantly change these results.

 This analysis shows that, asymptotically, the size of the hole in
 $\bk$-space that can be stably filled depends mostly on the product
 $\beta k_{0};$ in particular, it does not depend strongly on the
 extent of oversampling, provided that $m\geq 3.$ Perhaps most
 surprisingly, the norm of $\cR_{R,W}$ decreases with the dimension!
 Figure~\ref{fig13.04} and the tables in Example~\ref{exmpl3.01} show
 that the asymptotic formula provides a lower bound on
 $\|\cR_{R,W}\|$, which improves as $m,N$ increase.

\section{Appendix}\label{A2}

To prove Theorem~\ref{thm1}  we recall the variational characterizations of the singular
values of a linear map $A:\bbC^p\to\bbC^n$ and assume that $p\leq
n.$ It turns out to be simpler in the proof to list the singular values
in increasing order:
$ s_1\leq s_2\leq\cdots\leq s_p.$ Note that if $p> n,$
then $ s_1=0,$ which is not too interesting. The $j$th singular
value of $A$ has 2 variational characterizations:
\begin{equation}
  \begin{split}
   s^2_j&=\min_{\{S\subset \bbC^p:\: \dim S=j\}}\max_{\bx\in S:\:
    \bx\neq \bzero}\frac{\|A\bx\|^2}{\|\bx\|^2}\\
   s^2_j&=\max_{\{S\subset \bbC^p:\: \dim S=p-j+1\}}\min_{\bx\in S:\:
    \bx\neq \bzero}\frac{\|A\bx\|^2}{\|\bx\|^2}.
  \end{split}
\end{equation}

\begin{proof}
  The basic observation is that, because $\cF$ is a unitary map,
for $\bx\in \Im\pi_K,$ we have the identity
\begin{equation}
  \|\bx\|^2=\|\cF_{L,K}\bx\|^2+\|\cF_{L^c,K}\bx\|^2.
\end{equation}
We think of $\cF_{L,K}$ as a map from $\bbC^K$ to $\bbC^J,$ with
singular values $ s_1\leq\cdots \leq s_p,$ where $p=|K|,$ and
$\cF_{L^c,K}$ as a map from $\bbC^K$ to $\bbC^J,$ with singular values
$t_1\leq\cdots\leq t_p.$  Using
the observations above, we see that
\begin{equation}
  \begin{split}
   s_j^2&=\min_{\{S\subset \bbC^K:\: \dim S=j\}}\max_{\bx\in S:\:
    \bx\neq \bzero}\frac{\|\cF_{L,K}\bx\|^2}{\|\bx\|^2}\\
  &=\min_{\{S\subset \bbC^K:\: \dim S=j\}}\max_{\bx\in S:\:
    \bx\neq
    \bzero}\left[1-\frac{\|\cF_{L^c,K}\bx\|^2}{\|\bx\|^2}\right]\\
 &=1- \max_{\{S\subset \bbC^K:\: \dim S=j\}}\min_{\bx\in S:\:
    \bx\neq \bzero}\frac{\|\cF_{L^c,K}\bx\|^2}{\|\bx\|^2}\\
  &=1-t_{p-j+1}^2.
  \end{split}
\end{equation}
\end{proof}



  \noindent
  {\sc D.A.~Barmherzig:} dbarmherzig@flatironinstitute.org\newline
  {\sc A.H.~Barnett:} abarnett@flatironinstitute.org\newline
  {\sc C.L.~Epstein:} cle@math.upenn.edu\newline
  {\sc L.F.~Greengard:} lgreengard@flatironinstitute.org\newline
  {\sc J.F.~Magland:} jmagland@flatironinstitute.org\newline
  {\sc M.~Rachh:} mrachh@flatironinstitute.org
\end{document}